\documentclass[final,reqno]{elsarticle}
\usepackage{mathrsfs}
\usepackage{amssymb,amsmath,amsfonts,latexsym}
\usepackage{amsthm}
\usepackage{graphicx,float,epsfig,color,fancyhdr}
\usepackage{framed}
\usepackage{tikz} 
\usepackage[]{algorithm2e}
\usepackage{tabularx,colortbl}
\definecolor{LightCyan}{rgb}{0.88,1,1}
\definecolor{Gray}{gray}{0.85}

\newtheorem{lemma}{Lemma}[section]
\newtheorem{remark}{Remark}[section]
\newtheorem{proposition}{Proposition}[section]
\newtheorem{theorem}{Theorem}[section]

\newtheorem{problem}{Problem}

\def\CT{\mathcal{T}}
\def\E{K}
\def\G{\Gamma}

\def\HuO{H^1(\O)}

\def\Hcd{H_0^2(\O)}

\def\LO{L^2(\O)}
\def\N{\mathbb{N}}
\def\O{\Omega}
\def\P{\mathbb{P}}
\def\R{\mathbb{R}}
\def\dim{\mathop{\mathrm{\,dim}}\nolimits}

\def\Hcd{H_0^2(\O)}

\def\hdel{\widehat{\delta}}

\def\l{\lambda}

\def\sp{\mathop{\mathrm{sp}}\nolimits}
\def\sumkth{\sum_{K\in \mathcal{T}_h}}

\def\Vh{V_h}
\def\VK{V^{\E}_h}
\def\div{\mathop{\mathrm{\,div}}\nolimits}
\def\dite{\mathop{\bf{div}}\nolimits}

\def\qpt{\quad\forall}
\def\qqpt{\qquad\forall}

\def\HdoO{{H_0^2(\O)}}

\def\HdsO{{H^{2+s}(\O)}}
\def\HdoK{{H^{2}(\E)}}

\journal{}
\date{\today}

\def\ve{\mathop{\mathrm{\,v}}\nolimits}

\begin{document}

\begin{frontmatter}
\title{Virtual Element for the Buckling Problem of Kirchhoff-Love plates\\
\vspace{0.5cm}
{\it\normalsize Dedicated to Rodolfo Rodr\'iguez on his 65th birthday}}

\author[1,2]{David Mora}
\ead{dmora@ubiobio.cl}
\address[1]{GIMNAP, Departamento de Matem\'atica,
Universidad del B\'io-B\'io, Concepci\'on, Chile.}
\address[2]{CI$^2$MA, Universidad de Concepci\'on, Concepci\'on, Chile.}
\author[2,3]{Iv\'an Vel\'asquez}
\ead{ivelasquez@ing-mat.udec.cl}
\address[3]{Departamento de Ingenier\'ia Matem\'atica,
Universidad de Concepci\'on, Concepci\'on, Chile.}

\begin{abstract} 

In this paper, we develop a virtual
element method (VEM) of high order to solve the fourth
order plate buckling eigenvalue problem on polygonal meshes.
We write a variational formulation based on the Kirchhoff-Love
model depending on the transverse displacement of the plate.
We propose a $C^1$ conforming virtual element discretization
of arbitrary order $k\ge2$ and we use the so-called
Babu\v ska--Osborn abstract spectral approximation
theory to show that the resulting scheme
provides a correct approximation of the
spectrum and prove optimal order error estimates
for the buckling modes (eigenfunctions) and a double
order for the buckling coefficients (eigenvalues).
Finally, we report some numerical experiments
illustrating the behaviour of the proposed scheme
and confirming our theoretical results
on different families of meshes.

\end{abstract}


\end{frontmatter}


\setcounter{equation}{0}
\section{Introduction}
\label{SEC:INTR}

In this paper we analyze a conforming $C^1$ virtual element
approximation of an eigenvalue problem arising in Structural Mechanics:
the elastic stability of plates, in particular the so-called buckling problem.
This problem has attracted much interest since it is frequently encountered
in several engineering applications such as car or aircraft design. 
In particular, we will focus on thin plates which are modeled
by the Kirchhoff--Love equations.

The buckling problem for plates can be formulated as a spectral
problem of fourth order whose solution
is related with the limit of elastic stability of the plate
(i.e., eigenvalues-buckling coefficients and eigenfunctions-buckling modes).
This problem has been studied with several finite element methods,
for instance, conforming and non-conforming discretizations, mixed formulations.
We cite as a minimal sample of them \cite{BMS2014,CWCC2017,CKD,HLM2015,I2,MM2015,MoRo2009,Ra}.

The aim of the present paper is to introduce and analyze
a virtual element method (VEM) to solve the buckling problem.
The VEM has been introduced in \cite{BBCMMR2013} and
has been applied successfully in a large range of problems in fluid
and solid mechanics; see for
instance \cite{ABMV2014,ABSV2016,AMSLP2018,BBDMR2017,BBDMR2018,BBM,BLM2015,BLV-M2AN,BBBPS2016-0,
BGS17,CG2017,CGS17,CMS2016,ChM-camwa,MPP2018,PPR15,V-m3as18,WRR2016}.
Regarding VEM for spectral problems, we mention the following
recent works~\cite{BMRR,GMV2018,GV:IMA2017,MR2018,MRR2015,MRV,MV18}.

One important advantage of VEM is the possibility of easily implement highly
regular discrete spaces to solve fourth order partial differential equations \cite{ABSV2016,BM13,ChM-camwa}.
It is very well known that the construction of conforming finite elements
to $H^2$ is difficult in general, since they
generally involve a large number of degrees of freedom (see \cite{ciarlet}).
Here, we follow the VEM approach presented in \cite{BM13,ChM-camwa}
to build global discrete spaces of $C^1$ of arbitrary order that are
simple in terms of degrees of freedom and coding aspects to solve
an eigenvalue problem modelling the plate buckling problem.

More precisely, we will propose a $C^1$ Virtual Element Method
of arbitrary order $k\ge2$ to approximate the buckling
coefficients and modes of the plate buckling problem
on general polygonal meshes.
Based on the transverse displacements of the midplane of a thin
plate subjected to a symmetric stress tensor
field, we propose and analyze a variational formulation in $H^2$.
We characterize the continuous spectrum of the problem
through a certain continuous, compact and self-adjoint operator.
Then, we exploit the ability of VEM in order to construct highly
regular discrete spaces and propose a conforming discretization of
the buckling eigenvalue problem in $H^2$ which is an extension
of the discrete virtual space introduced in \cite{ABSV2016,BM13}.
We construct projection operators in order to
write bilinear forms that are fully computable.
In particular, to discretize the right hand side of
the eigenvalue problem we propose a simple bilinear
form which does not need any stabilization.
This make possible to use directly the so-called Babu\v ska--Osborn
abstract spectral approximation theory (see \cite{BO}) to show that
under standard shape regularity assumptions the resulting
virtual element scheme provides a correct approximation of the
spectrum and prove optimal order error estimates for the
eigenfunctions and a double order for the eigenvalues.
The proposed VEM method provides an attractive and competitive
alternative to solve the fourth order plate buckling
eigenvalue problem in term of its computational cost.
For instance, in the lowest order configuration $(k=2)$,
the computational cost is almost $3N_v$, where $N_v$ denotes 
the number of vertices in the polygonal mesh. For $k=3$,
the computational cost is almost $3N_v+N_e$,
where $N_e$ denotes the number of edges in the polygonal mesh.

This paper is structured as follows: In Section~\ref{SEC:SpPrCont},
we present the variational formulation for the plate buckling eigenvalue problem.
We define a solution operator whose spectrum allows us to characterize
the spectrum of the buckling problem. In Section~\ref{SEC:DISCRETE}
we introduce the virtual element discretization of arbitrary degree $k\ge2$,
describe the spectrum of a discrete solution operator and prove
some auxiliary results. In Section~\ref{SEC:approximation},
we prove that the numerical scheme presented in this work provides
a correct spectral approximation and establish optimal order error
estimates for the eigenvalues and eigenfunctions. Finally,
in Section~\ref{SEC:Num-Res} we report some numerical tests
that confirm the theoretical analysis developed. 

Throughout the article we will use standard notations for Sobolev
spaces, norms and seminorms. Moreover, we will denote by $C$ a
generic constant independent of the mesh parameter $h$, which
may take different values in different occurrences.


\setcounter{equation}{0}
\section{Presentation of the continuous spectral problem.}
\label{SEC:SpPrCont}
Let $\O\subseteq \R^2$ be a polygonal bounded domain
corresponding to the mean surface of a plate in its
reference configuration. The plate is assumed to be homogeneous,
isotropic, linearly elastic, and sufficiently thin as to
be modeled by Kirchhoff-Love equations.
The buckling eigenvalue problem of a clamped plate, which is subjected
to a plane stress tensor field $\boldsymbol{\eta}:\O\to\R^{2\times 2}$
with $\boldsymbol{\eta}\neq 0$ reads as follows:
\begin{equation}\label{cont_pro}
\left\{\begin{array}{lll}
\Delta^2 u =-\lambda \div (\boldsymbol{\eta}\nabla u)\quad &\mbox{in }\O,\\
u=\partial_\nu u=0\quad &\mbox{on }\G.
\end{array}\right.
\end{equation}
The unknowns of this eigenvalue problem are the deflection of
the plate $u$ (buckling modes) and the eigenvalue $\lambda$
(scaled buckling coefficients).
We have denoted by $\partial_\nu$ the normal derivative.
To simplify the notation we have taken the Young modulus
and the density of the plate, both equal to 1. In addition,
the stress tensor field is assumed to satisfy the following equilibrium equations:
\begin{align}
\boldsymbol{\eta}^{\rm t}=\boldsymbol{\eta}&\quad\mbox{in }\O, \nonumber\\
\dite \boldsymbol{\eta} =0& \quad \mbox{in }\O.\nonumber
\end{align}

In the remain of this section and in Section~\ref{SEC:DISCRETE},
it is enough to consider
$\boldsymbol{\eta}\in L^{\infty}(\O)^{2\times2}$.
However, we will assume some additional regularity
which will be used in the proof of Theorem~\ref{orderconv}.
In addition, we do not need to assume $\boldsymbol{\eta}$
to be positive definite. Let us remark that,
in practice, $\boldsymbol{\eta}$ is the stress
distribution on the plate subjected to in-plane loads,
which does not need to be positive definite \cite{Timoshenko1961}.

\subsection{The continuous formulation.}
In this section we will present and analyze a variational
formulation associated with the spectral problem.
We will also introduce the so-called solution operator
whose spectra will be related to the solutions of the continuous
spectral problem \eqref{cont_pro}.

In order to write the variational formulation of the
spectral problem, we introduce the following symmetric bilinear forms in $\Hcd$:
\begin{align*}
a(u,v)
:=\int_{\O}D^2 u:\,D^2 v, \qquad
b(u,v)
:=\int_{\O}(\boldsymbol{\eta} \nabla u)\cdot \nabla v,
\end{align*}
where $":"$ denotes the usual scalar product of $2\times2$-matrices,
$D^2 v:=(\partial_{ij}v)_{1\le i,j\le2}$ denotes the Hessian
matrix of $v$. It is easy to see that $a(\cdot,\cdot)$ 
is an inner-product in $\Hcd$.

The variational formulation of the eigenvalue problem \eqref{cont_pro} is given as follows:
\begin{problem}\label{Prob1_cont_buckling}
Find $(\l,u)\in\R\times\HdoO$, $u\ne0$, such that
\begin{equation}\label{vibrapri}
\begin{array}{llll}
a(u,v) =\l b(u,v) 
\qqpt v\in\HdoO.
\end{array}
\end{equation}
\end{problem}

The following result establishes that the
bilinear form $a(\cdot,\cdot)$ is elliptic in $\Hcd$.
\begin{lemma}
\label{ha-elipt}
There exists a constant $\alpha_{0}>0$, depending on $\O$, such that
$$
a(v,v)
\ge\alpha_{0}\left\|v\right\|_{2,\O}^2
\qquad\forall v\in\HdoO.
$$
\end{lemma}

\begin{proof}
The result follows immediately from the fact that
$\Vert D^2 v\Vert_{0,\O}$ is a norm on $\HdoO$,
equivalent with the usual norm. 
\end{proof}

\begin{remark}\label{nozero}
We have that $\lambda\ne0$ in problem \eqref{vibrapri}.
Moreover, it is easy to prove, using the symmetry of $\boldsymbol{\eta}$,
that all the eigenvalues are real (not necessarily positive).
We also have that $b(u,u)\ne0$.
\end{remark}

Next, in order to analyze the variational eigenvalue problem \eqref{vibrapri},
we introduce the following solution operator:
\begin{align*}
T:\ \HdoO & \longrightarrow \HdoO,
\\
f & \longmapsto Tf:=w,
\end{align*}
where $w\in\HdoO$ is the unique solution (as a consequence of Lemma~\ref{ha-elipt})
of the following source problem:
\begin{equation}
\label{T1}
a(w,v)=b(f,v)
\qquad\forall v\in\HdoO.
\end{equation}

We have that the linear operator $T$ is well
defined and bounded. Notice that $(\l,u)\in\R\times\HdoO$ solves
problem~\eqref{vibrapri} if and only if $Tu=\mu u$
with $\mu\neq0$ and $u\ne0$, in which case $\mu:=\frac1{\l}$.
In addition, using the symmetry of $\boldsymbol{\eta}$, we
can deduce that $T$ is self-adjoint with respect to
the inner product $a(\cdot,\cdot)$ in $\HdoO$.
Indeed, given $f,g\in\HdoO$, 
$$
a(Tf,g)
=b(f,g)
=b(g,f)
=a(Tg,f)
=a(f,Tg).
$$

On the other hand, the following is an additional
regularity result for the solution of problem~\eqref{T1}
and consequently, for the eigenfunctions of $T$.

\begin{lemma}\label{LEM:REG_buck}
There exists $s_{\Omega}>1/2$ such that the following results hold: 
\begin{itemize}
\item[(i)] 
For all $f\in\HuO$, there exists a positive constant
$C>0$ such that any solution $w$ of the source problem~\eqref{T1}
satisfies $w\in H^{2+\tilde{s}}(\O)$ with $\tilde{s}:=\min\{s_{\Omega},1\}$ and
\begin{equation*}
\Vert w\Vert_{2+\tilde{s},\O}\le C\Vert f\Vert_{1,\O}.
\end{equation*}
\item[(ii)] If $(\lambda,u)$ is an eigenpair of the spectral
problem~\eqref{vibrapri}, there exist $s>1/2$ and a positive
constant $C$ depending only on $\Omega$ such that $u\in H^{2+s}(\O)$ and
\begin{equation*}
\Vert u\Vert_{2+s,\Omega}\leq C\Vert u\Vert_{2,\Omega}.
\end{equation*}
\end{itemize}
\end{lemma}

\begin{proof}
	The proof follows from the classical regularity result for the
	biharmonic problem with its right-hand side in $\LO$ (cf. \cite{G}).
\end{proof}

Therefore, because of the compact inclusion $\HdsO\hookrightarrow\HdoO$,
$T$ is a compact operator. Thus, we finish this section with the
following spectral characterization result.

\begin{lemma}
\label{CHAR_SP_buck}
The spectrum of $T$ satisfies
$\sp(T)=\{0\}\cup\left\{\mu_k\right\}_{k\in\N}$, where
$\left\{\mu_k\right\}_{k\in\N}$ is a sequence of real
eigenvalues which converges to $0$. The multiplicity of each eigenvalue is finite.
\end{lemma}

\setcounter{equation}{0}
\section{Spectral approximation.}
\label{SEC:DISCRETE}

In this section, we will write a VEM discretization
of the spectral problem \eqref{vibrapri}.
With this aim, we start with the mesh construction
and the assumptions considered to introduce the discrete
virtual element spaces.

Let $\left\{\CT_h\right\}_h$ be a sequence of decompositions of $\O$
into polygons $\E$ we will denote by $h_\E$ the diameter of the element $\E$
and $h$ the maximum of the diameters of all the elements of the mesh,
i.e., $h:=\max_{\E\in\CT_h}h_\E$.
In what follows, we denote by $N_\E$ the number of vertices of $\E$,
by $e$ a generic edge of $\left\{\CT_h\right\}_h$ and for all $e\in\partial\E$,
we define a unit normal vector $\nu_\E^e$ that points outside of $\E$.

In addition, we will make the following
assumptions as in \cite{BBCMMR2013,BMRR}:
there exists a positive real number $C_{\CT}$ such that,
for every $h$ and every $\E\in \CT_h$,
\begin{itemize}
\item[{\bf A1}:] $\E\in\CT_h$ is star-shaped with
respect to every point of a  ball
of radius $C_{\CT}h_\E$;
\item[{\bf A2}:] the ratio between the shortest edge
and the diameter $h_\E$ of $\E$ is larger than $C_{\CT}$.
\end{itemize}

In order to introduce the discretization, for every integer
$k\ge2$ and for every polygon $\E$, we define the
following finite dimensional space:
\begin{align*}
\widetilde{V}_h^K
:=\left\{v_h\in \HdoK : \Delta^2v_h\in\P_{k-2}(\E), v_h|_{\partial\E}\in C^0(\partial\E),
v_h|_e\in\P_r(e)\,\,\forall e\in\partial\E,\right.\\
\left.\nabla v_h|_{\partial\E}\in C^0(\partial\E)^2,
\partial_{\nu_K^e} v_h|_e\in\P_s(e)\,\,\forall e\in\partial\E\right\},
\end{align*}
where $r:=\max\{3,k \}$ and $s:=k-1$.

This space has been recently considered in \cite{ChM-camwa}
to obtain optimal error estimates for fourth order PDEs
and it can be seen as an extension of the $C^1$ virtual space
introduced in \cite{BM13} to solve the bending problem of thin plates.
Here, we will consider the same space together with
an enhancement technique (cf. \cite{AABMR13}) to build a computable right hand
of the buckling eigenvalue problem.

It is easy to see that any $v_h \in \widetilde{V}_h^K$
satisfies the following conditions:
\begin{itemize}
	\item the trace (and the trace of the gradient)
	on the boundary of $\E$ is continuous;	
	\item $\P_k(\E)\subseteq\widetilde{V}_h^K$.
\end{itemize}

In $\widetilde{V}_h^K$ we define the following five sets of linear operators.
For all $v_{h}\in\widetilde{V}_h^K$:
\begin{itemize}
\item[${\bf D_1}$:] evaluation of $v_{h}$ at the $N_{\E}$ vertices of $\E$;
\item[${\bf D_2}$:]  evaluation of $\nabla v_h$ at the $N_{\E}$ vertices of $\E$;
\item[${\bf D_3}$:]  {For } $r>3$, \mbox{ the moments } $\int_e q(\xi)v_h(\xi)d\xi\qquad \quad \forall q\in \mathbb{P}_{r-4}(e),
\quad\forall  \mbox{ edge  } e;\label{r_moments}$
\item[${\bf D_4}$:]  {For } $s>1$, \mbox{ the moments } $\int_e q(\xi)\partial_{\nu_K^{e}}v_h(\xi)d\xi\quad \ \forall q \in \mathbb{P}_{s-2}(e),
\quad\forall \mbox{ edge } e;\label{s_moments}$
\item[${\bf D_5}$:] {For } $k\geq 4$,  \mbox{ the moments } $\int_K q(\boldsymbol{x})v_h(\boldsymbol{x})d\boldsymbol{x}\qquad   \forall q\in \mathbb{P}_{k-4}(K),
\ \ \forall \mbox{ polygon } K \label{k_moments}$.
\end{itemize}

In order to construct the discrete scheme,
we need some preliminary definitions.
First, we note that bilinear form $a(\cdot,\cdot)$,
introduced in the previous section,
can be split as follows:
$$
a(u,v)=\sum_{\E\in\CT_h}a_{\E}(u,v),
\qquad u,v\in\HdoO,
$$
with
\begin{equation*}
\label{alocal}
a_{\E}(u,v)
:=\int_{\E}D^2 u:\,D^2 v,
\qquad u,v\in\HdoK.
\end{equation*}

Now, we define the projector
$\Pi_K^{k,D}:H^2(K)\to \mathbb{P}_k(K)\subseteq\widetilde{V}_h^K$
as the solution of the following local problems (in each element $\E$): 
\begin{subequations}
	\begin{align}
	& a_{\E}\big(\Pi_K^{k,D} v,q\big) =a_{\E}(v,q)\qquad\forall q\in\P_k(\E)\quad \forall v\in H^2(K),\label{def_Pi_KkD}\\
	&\widehat{\Pi_K^{k,D} v}=\widehat{v}, \quad \widehat{\nabla \Pi_K^{k,D} v}=\widehat{\nabla v},\label{kerPi_KkD}
	\end{align}
\end{subequations}
where $\widehat{v}$ is defined as follows:
\begin{align*}
\widehat{v}:=\frac{1}{N_K}\sum_{i=1}^{N_K}v(\ve_i)\qquad\forall v\in C^0(\partial\E)
\end{align*}
and $\ve_i, 1\le i\le N_{\E}$, are the vertices of $\E$.

We observe that bilinear form $a_{\E}(\cdot,\cdot)$
has a non-trivial kernel given by $\P_1(\E)$. Hence,
the role of condition \eqref{kerPi_KkD} is to select
an element of the kernel of the operator.

It is easy to see that operator $\Pi_{K}^{k,D}$
is well defined on $\widetilde{V}_h^K$.
Moreover, the following result states that
for all $v\in\widetilde{V}_h^K$ the polynomial
$\Pi_{K}^{k,D} v$ can be computed using the output
values of the sets ${\bf D_1}-{\bf D_5}$.

\begin{lemma}
The operator $\Pi_{K}^{k,D}:\widetilde{V}_h^K\to \mathbb{P}_k(K)$ is  explicitly computable
for every $v\in\widetilde{V}_h^K$, using only the information of the
linear operators in ${\bf D_1}-{\bf D_5}$.
\end{lemma}
\begin{proof}
For all $v_h\in\widetilde{V}_h^K$ we integrate twice by parts on
the right-hand side of \eqref{def_Pi_KkD}. We obtain 
\begin{align}
a(v_h,q)&=\int_K D^2v_h:D^2q\nonumber\\
&=\int_K \Delta^2 q v_h - \int_{\partial K} \dite (D^2 q)\cdot \nu_K v_h
+\int_{\partial K}D^2q \nu_K \cdot \nabla v_h.\label{Pi_KDcomp}
\end{align}
It is easy to see that since $\Delta^2 q \in \mathbb{P}_{k-4}(K)$
hence the first integral in the right-hand side of \eqref{Pi_KDcomp}
is computable using the output values of the set ${\bf D_5}$.
We also note that the boundary integrals of \eqref{Pi_KDcomp}
only depends on the boundary values of $v_h$ and $\nabla v_h$,
so they are computable using the output values of the sets ${\bf D_1}-{\bf D_4}$.
On the other hand, the kernel part of $\Pi_{K}^{k,D}$ (cf. \eqref{kerPi_KkD})
is computable using the output values of the sets ${\bf D_1}-{\bf D_2}$.
\end{proof}

We introduce our local virtual space:
\begin{align*}
\VK:=\left\{v_h\in\widetilde{V}_h^K : \int_{\E}(\Pi_{K}^{k,D} v_h)q
=\int_{\E}v_hq\qquad\forall q\in\P_{k-3}^{*}(\E)\cup\P_{k-2}^{*}(\E)\right\}.
\end{align*}
where $\P_{\ell}^{*}(\E)$ denotes homogeneous polynomials
of degree $\ell$ with the convention that $\P_{-1}^{*}(\E)=\{0\}$.

Note that $V_h^K\subseteq \widetilde{V}_h^K$.
Thus, the linear operator $\Pi_{K}^{k,D}$ is well
defined on $V_h^K$ and computable only using the output values
of the sets ${\bf D_1}-{\bf D_5}$.
We also have that $\mathbb{P}_k(K)\subseteq\VK$.
This will guarantee the good approximation properties of the space.

Moreover, it has been established in \cite{ChM-camwa}
that the set of linear operators ${\bf D_1}-{\bf D_5}$
constitutes a set of degrees of freedom for $V_h^K$.

Now, we consider the $L^2(K)$ orthogonal projector
onto $\P_{k-2}(\E)$ as follows:
we define $\Pi_{K}^{k-2}:L^2(K)\to\P_{k-2}(\E)$ for each $v \in L^2(K)$ by
\begin{equation}\label{fff}
\int_{\E}(\Pi_{K}^{k-2} v)q=\int_{\E}vq\qquad\forall q\in\P_{k-2}(\E).
\end{equation}
Next, due to the particular property appearing in definition
of the space $\VK$, it can be seen that the right hand
side in \eqref{fff} is computable using $\Pi_{K}^{k,D} v$,
and the degrees of freedom given by ${\bf D_5}$
and thus $\Pi_{K}^{k-2} v$ depends only on the values of
the degrees of freedom given by ${\bf D_1}-{\bf D_5}$ when $v\in \VK$.

In order to discretize the right hand side of the buckling eigenvalue
problem, we will consider the following
projector onto $\P_{k-1}(\E)^2$:
we define $\boldsymbol{\Pi}_K^{k-1}:H^1(K)\to\P_{k-1}(\E)^2$ for each $v \in H^1(K)$ by
\begin{equation*}
\int_{\E}(\boldsymbol{\Pi}_K^{k-1}\nabla v)\cdot {\bf q}=\int_{\E}\nabla v\cdot {\bf q}
\quad\forall {\bf q}\in\P_{k-1}(\E)^2.
\end{equation*}

In addition, we observe that the for any $v_h\in \VK$, the vector function
$\boldsymbol{\Pi}_K^{k-1}\nabla v_h$ can be explicitly
computed from the degrees of freedom ${\bf D_1}-{\bf D_5}$.
In fact, in order to compute $\boldsymbol{\Pi}_K^{k-1}\nabla v_h$,
for all $\E\in\CT_h$ we must be able to calculate the following:
\begin{equation*}
\int_{\E}\nabla v_h\cdot{\bf q}\qquad\forall {\bf q}\in\P_{k-1}(\E)^2.
\end{equation*}
From an integration by parts, we have
\begin{equation*}
\begin{split}
\int_{\E}\nabla v_h\cdot{\bf q}&=-\int_{\E}v_h\div{\bf q}+\int_{\partial\E}v_h({\bf q}\cdot \nu_K)
\qquad\forall {\bf q}\in\P_{k-1}(\E)^2,\\
&=-\int_{\E}\Pi_{K}^{k-2} v_h\div{\bf q} +\int_{\partial\E}v_h({\bf q}\cdot \nu_K)
\qquad\forall {\bf q}\in\P_{k-1}(\E)^2.
\end{split}
\end{equation*}
The first term on the right-hand side above depends only on 
the $\Pi_{K}^{k-2} v_h$ and this depends on the values of
the degrees of freedom ${\bf D_1}-{\bf D_5}$ (cf. \eqref{fff}).
The second term can also be computed since ${\bf q}$ is a polynomial
of degree $k-1$ on each edge and therefore
is uniquely determined by the values of ${\bf D_1}-{\bf D_5}$.

Now, we are ready to define our global virtual space
to solve the plate buckling eigenvalue problem, this is defined as follows:

\begin{equation}\label{glob_virt_spac_buck}
V_h:=\Big\{v_h\in \Hcd: v_h|_K\in V_h^K  \Big\}.
\end{equation}

In what follows, we discuss the construction
of the discrete version of the local forms.
With this aim, we consider 
$s_{K}^{D}(\cdot,\cdot)$ any symmetric positive
definite and computable bilinear form to be chosen as to satisfy:
\begin{equation}\label{term-stab-SK}
c_0 a_K(v_h,v_h)\leq s_\E^{D}(v_h,v_h)\leq
c_1 a_K(v_h,v_h)\quad \forall v_h \in V_h^K\quad \mbox{with }\quad \Pi_K^{k,D} v_h=0.
\end{equation}

Then, we set
\begin{align*}
&a_h(u_h,v_h)
:=\sum_{\E\in\CT_h}a_{h,\E}(u_h,v_h),
\qquad u_h,v_h\in\Vh,\\
&b_h(u_h,v_h)
:=\sum_{\E\in\CT_h}b_{h,\E}(u_h,v_h),
\qquad u_h,v_h\in\Vh, 
\end{align*}
with $a_{h,\E}(\cdot,\cdot)$ and  $b_{h,\E}(\cdot,\cdot)$   
are the local bilinear forms on
$\VK\times\VK$ defined by
\begin{align}
& a_{h,\E}(u_h,v_h)
:=a_{\E}\big(\Pi_K^{k,D} u_h,\Pi_K^{k,D} v_h\big)
+s_{\E}^{D}\big(u_h-\Pi_K^{k,D} u_h,v_h-\Pi_K^{k,D} v_h\big), \label{locforma1}\\[1ex]
& b_{h,\E}(u_h,v_h)
:=\int_K \boldsymbol{\eta} \boldsymbol{\Pi}_K^{k-1}\nabla  u_h\cdot \boldsymbol{\Pi}_K^{k-1} \nabla  v_h. \label{locforma4} 
\end{align}

Notice that the bilinear form $s_{\E}^{D}(\cdot,\cdot)$ has to be
actually computable for $u_h,v_h\in \VK$.

\begin{proposition}
The local bilinear form $a_{h,\E}(\cdot,\cdot)$ on each element $\E$ satisfy
\begin{itemize}
\item \textit{Consistency}: for all $h > 0$ and for all $\E\in\CT_h$, we have that
\begin{align}
a_{h,\E}(p,v_h)
=a_{\E}(p,v_h)
\qquad\forall p\in\P_k(\E),
\quad\forall v_h\in\VK,\label{consis-a} 
\end{align}
\item \textit{Stability and
boundedness}: There exist two positive constants
$\alpha_1,\alpha_2$, independent of $\E$, such that:
\begin{align}
\alpha_1 a_{\E}(v_h,v_h) &\leq a_{h,\E}(v_h,v_h) \leq\alpha_2 a_{\E}(v_h,v_h)&
\qquad\forall v_h\in\VK.\label{stab-a}
\end{align}
\end{itemize}
\end{proposition}

\subsection{The discrete eigenvalue problem.}
Now, we are in a position to write the virtual
element discretization of Problem~\ref{Prob1_cont_buckling} as follows. 
\begin{problem}\label{Prob1disc_cont_buckling}
Find $(\l_h,u_h)\in\R\times\Vh$, $u_h\ne0$, such that
\begin{equation}\label{Prob1h_buckling}
a_h(u_h,v_h)=\l_h b_h(u_h,v_h)
\qquad\forall v_h\in\Vh.
\end{equation}
\end{problem}

We observe that by virtue of \eqref{stab-a}, the
bilinear form $a_{h}(\cdot,\cdot)$ is bounded. Moreover, as shown in
the following lemma, it is also uniformly elliptic.
\begin{lemma}
\label{ha-elipt-disc}
There exists a constant $\alpha>0$, independent of $h$, such that
$$
a_{h}(v_h,v_h)
\ge\alpha\left\|v_h\right\|_{2,\O}^2
\qquad\forall v_h\in\Vh.
$$
\end{lemma}
\begin{proof}
Thanks to \eqref{stab-a} and Lemma~\ref{ha-elipt}, it is easy to check that
the above inequality holds with
$\alpha:=\alpha_{0}\min\left\{\alpha_{1},1\right\}$.
\end{proof}

In order to analyze the discrete problem, we introduce the
solution operator associated to Problem~2 as follows:
\begin{align*}
T_h: \Hcd & \longrightarrow \Hcd,
\\
f & \longmapsto T_h f:=w_h,
\end{align*}
with $w_h$ the unique solution of the following source problem
\begin{equation}
a_h(w_h,v_h)=b_h(f,v_h)\qquad \forall v_h\in \Vh. \label{defTh_b}
\end{equation}

Note that the ellipticity of $a_h(\cdot,\cdot)$ established in Lemma~\ref{ha-elipt-disc},
the boundedness of the right hand side (cf. \eqref{locforma4}) and
Lax-Milgram Lemma guarantee that $T_h$ is well defined.
Moreover, as in the continuous case, $(\l_h,u_h)\in\R\times\Vh$
solves problem~\eqref{Prob1h_buckling} 
if and only if $T_hu_h=\mu_h u_h$ with
$\mu_h\neq0$ and $u_h\ne0$, in which case $\mu_h:=\frac1{\l_h}$.

\begin{remark}
The same arguments leading to Remark~\ref{nozero}
allow us to show that any solution of \eqref{Prob1h_buckling}
satisfies $\lambda_h\ne0$. Moreover, $b_h(u_h,u_h)\ne 0$ also
holds true.
\end{remark}

Moreover from the definition of $a_h(\cdot,\cdot)$ and $b_h(\cdot,\cdot)$
we can check that $T_h$ is self-adjoint with respect to inner product
$a_h(\cdot,\cdot)$. Therefore, we can describe the spectrum
of the solution operator $T_h$.

Now, we are in position to write the following characterization of the
spectrum of the solution operator.

\begin{theorem}
\label{CHAR_SP_DISC}
The spectrum of $T_h$ consists of $M_h:=\dim(\Vh)$ eigenvalues,
repeated according to their respective multiplicities.
The spectrum decomposes as follows: $\sp(T_h)=\{0\}\cup\{\mu_h\}_{k=1}^{\kappa}$,
where $\kappa=M_h-\dim Z_h$ with $Z_h:=\left\{u_h\in\Vh:\, b_h(u_h,v_h)=0\qpt v_h\in\Vh\right\}$.
The eigenvalues $\mu_h$ are all real and non-zero.
\end{theorem}

\setcounter{equation}{0}
\section{Convergence and error estimates.}
\label{SEC:approximation}

In this section we will establish convergence
and error estimates of the proposed VEM discretization.
With this aim, we will prove that $T_h$ provides
a correct spectral approximation of $T$ using the
classical theory for compact operators (see \cite{BO}).

We start with the following approximation result,
on star-shaped polygons,
which is derived by interpolation
between Sobolev spaces (see for instance \cite[Theorem I.1.4]{GR}
from the analogous result for integer values of $s$).
We mention that this result has been stated
in \cite[Proposition 4.2]{BBCMMR2013} for integer values
and follows from the classical Scott-Dupont theory (see \cite{BS-2008}
and \cite[Proposition 3.1]{ABSV2016}):

\begin{proposition}
\label{app1}
There exists a constant $C>0$,
such that for every $v\in H^{\delta}(\E)$ there exists
$v_{\pi}\in\P_k(\E)$, $k\geq 0$ such that
\begin{equation*}
\vert v-v_{\pi}\vert_{\ell,\E}\leq C h_\E^{\delta-\ell}|v|_{\delta,\E}\quad 0\leq\delta\leq
k+1, \ell=0,\ldots,[\delta],
\end{equation*}
with $[\delta]$ denoting largest integer equal or smaller than $\delta \in {\mathbb R}$.
\end{proposition}

In what follows, we write several auxiliary results which will be useful in the
forthcoming analysis. First, we write standard error estimations
for the projector $\boldsymbol{\Pi}_K^{k-1}$.

\begin{lemma}
\label{app3}
There exists $C>0$ independent of $h$ such that for all ${\bf v}\in H^{\delta}(K)^2$
\begin{equation*}
\Vert {\bf v}-\boldsymbol{\Pi}_K^{k-1}{\bf v}\Vert_{0,\E}\leq
C h_\E^{\delta}|{\bf v}|_{\delta,\E}\quad 0\leq \delta\leq
k+1.
\end{equation*}
\end{lemma}

Now, we present an interpolation result in the virtual space
$V_h$ (see \cite{ABSV2016,BMR2016}).

\begin{proposition}\label{app2}
Assume {\textbf{A1}--\textbf{A2}} are satisfied, then
for all $v\in H^s(K)$ there exist $v_I\in V_h$ and $C>0$	independent of $h$ such that
\begin{equation*}
||v-v_I ||_{l,K}\leq C h_K^{s-l}|v|_{s,K},\quad l=0,1,2, \quad 2\leq s\leq k+1. 
\end{equation*}
\end{proposition}

Now, in order to prove the convergence of our method,
we introduce the following broken $H^{s}$-seminorm ($s=1,2$):
$$|v|_{s,h}:=\Big(\sum_{\E\in\CT_h}|v|_{s,\E}^{2}\Big)^{1/2},$$
which is well defined for every $v\in L^{2}(\O)$ such that
$v|_{\E}\in H^{s}(\E)$ for all polygon $\E\in \CT_{h}$.

Now, with these definitions we have the following results.

\begin{lemma}
\label{lemcotste}
There exists $C>0$ such that, for all $f\in\Hcd$, if $w=Tf$ and
$w_h=T_h f$, then
$$
\left\|\left(T-T_h\right)f\right\|_{2,\O}
=\left\|w-w_h\right\|_{2,\O}
\le C\Big(h||f ||_{2,\O}
+\left\|w-w_{I}\right\|_{2,\O}
+\vert w-w_{\pi}\vert_{2,h}\Big),
$$
for all $w_I\in\Vh$ and for all $w_{\pi}\in\LO$ such that
$w_{\pi}|_{\E}\in\P_k(\E) \quad \forall \E\in\CT_h$.
\end{lemma}

\begin{proof}
Let $f\in \Hcd$, and $w=Tf$ and $w_h=T_h f$.
For $w_I\in\Vh$, we set $v_h:=w_h-w_I$. Thus
\begin{equation}\label{convThplusminusu_I}
|| (T-T_h)f||_{2,\O}\leq || w-w_I||_{2,\O} +||v_h ||_{2,\O}.
\end{equation}
Now, thanks to Lemma~\ref{ha-elipt-disc}, the definition
of $a_{h,{\E}}(\cdot,\cdot)$ and those of $T$ and $T_h$, we have
\begin{align}
\alpha || v_h||_{2,\O}^2& \leq a_h(v_h,v_h)=a_h(w_h,v_h)-a_h(w_I,v_h)=a_h(w_h,v_h)-\sum\limits_{K\in \mathcal{T}_h}a_{h,K}(w_I,v_h)\nonumber\\
&  =a_h(w_h,v_h)-\sum\limits_{K\in \mathcal{T}_h} \Big\{ a_{h,K}(w_I-w_{\pi},v_h) +a_{h,K}(w_{\pi},v_h) \Big\}\nonumber\\ 
&  = a_h(w_h,v_h)-\sum\limits_{K\in \mathcal{T}_h} \Big\{ a_{h,K}(w_I-w_{\pi},v_h) +a_K(w_\pi-w,v_h)+a_{K}(w,v_h) \Big\}\nonumber\\
&  =a_h(w_h,v_h)-a(w,v_h)-\sum\limits_{K\in \mathcal{T}_h} \Big\{ a_{h,K}(w_I-w_{\pi},v_h) +a_K(w_\pi-w,v_h) \Big\}\label{eqconv2}.
\end{align}

We bound each term on the right hand side of \eqref{eqconv2}. The first term can be estimated as follows
\begin{align*}
& a_h(w_h,v_h)-a(w,v_h)= b_h(f,v_h)-b(f,v_h) \nonumber\\
&=\sumkth\Bigg\{\int_K\big\{ \boldsymbol{\eta}\boldsymbol{\Pi}_K^{k-1}\nabla f\cdot \boldsymbol{\Pi}_K^{k-1}\nabla v_h- \boldsymbol{\eta}\nabla f\cdot \nabla v_h\big\} \Bigg \}\nonumber\\
&=\sumkth\Bigg\{\int_K\big\{ \boldsymbol{\eta}\boldsymbol{\Pi}_K^{k-1}\nabla f\cdot \boldsymbol{\Pi}_K^{k-1}\nabla v_h - \boldsymbol{\eta}\nabla f\cdot \boldsymbol{\Pi}_K^{k-1}\nabla v_h + \boldsymbol{\eta}\nabla f\cdot \boldsymbol{\Pi}_K^{k-1}\nabla v_h - \boldsymbol{\eta}\nabla f\cdot \nabla v_h\big\} \Bigg \}\nonumber\\
&=\sumkth\Bigg\{\int_K\big\{ \boldsymbol{\eta}\left(\boldsymbol{\Pi}_K^{k-1}\nabla f-\nabla f\right)\cdot \boldsymbol{\Pi}_K^{k-1}\nabla v_h + \boldsymbol{\eta}\nabla f\cdot \left(\boldsymbol{\Pi}_K^{k-1}\nabla v_h -\nabla v_h\right)\big\} \Bigg \}\nonumber\\
&\le \sumkth C\Bigg\{ \Vert \boldsymbol{\Pi}_K^{k-1}\nabla f-\nabla f\Vert_{0,K}\Vert\boldsymbol{\Pi}_K^{k-1}\nabla v_h\Vert_{0,K}+\Vert\nabla f\Vert_{0,K}\Vert\boldsymbol{\Pi}_K^{k-1}\nabla v_h-\nabla v_h\Vert_{0,K}\Bigg \}\nonumber\\
&\le Ch\Vert f\Vert_{2,\Omega}\Vert v_h\Vert_{2,\Omega},
\end{align*}
where we have used Lemma~\ref{app3} in the last inequality.
Notice taht the constant
$C>0$ depends on $\Vert\boldsymbol{\eta}\Vert_{\infty}$.

Next, using the stability of $a_{h,K}(\cdot,\cdot)$,
the Cauchy-Schwarz and triangular inequalities in the
second term on the right hand side of \eqref{eqconv2},  we have

\begin{align*}
\alpha\left\|v_h\right\|^2_{2,\O}
&\le C\Big(h||f ||_{2,\O}
+\left\|w-w_{I}\right\|_{2,\O}
+\vert w-w_{\pi}\vert_{2,h}\Big) \Vert v_h\Vert_{2,\O}.
\end{align*}

Thus, the result follows from the previous bounds together with \eqref{eqconv2}.
\end{proof}

Now we are in a position to prove that the
operator $T_h$ converges in norm to $T$. 

\begin{theorem}\label{conv_norm_buckling}
For all $f\in \Hcd$, there exist $\tilde{s}\in(\frac{1}{2},1]$
and $C>0$ independent of $h$ such that 
$$||(T-T_h)f||_{2,\O}\le C h^{\tilde{s}}||f||_{2,\O}.$$
\end{theorem}
\begin{proof}
The proof is obtained from Lemma~\ref{lemcotste}
and Propositions~\ref{app1} and \ref{app2} and Lemma~\ref{LEM:REG_buck}.
\end{proof}

Next, we will use the classical theory for compact operators
(see \cite{BO} for instance) in order to prove convergence and error estimates
for eigenfunctions and eigenvalues. Indeed, an immediate consequence of
Theorem~\ref{conv_norm_buckling} is that isolated parts of $\sp(T)$ are approximated
by isolated parts of $\sp(T_h)$.
It means that if $\mu$ is a nonzero eigenvalue of $T$ with algebraic
multiplicity $m$, hence there exist $m$ eigenvalues 
$\mu_h^{(1)},\ldots,\mu_h^{(m)}$ of $T_h$
(repeated according to their respective multiplicities)
that will converge to $\mu$ as $h$ goes to zero. 

Now, let us denote by $\mathcal{E}$ and $\mathcal{E}_h$ the
eigenspace associated to the eigenvalue $\mu$ and the spanned
of the eigenspaces associated to $\mu_h^{(1)},...,\mu_h^{(m)}$, respectively.

We also recall the definition of the \textit{gap} $\hdel$ between two closed
subspaces $\mathcal{X}$ and $\mathcal{Y}$ of a Hilbert space $\mathcal{V}$:
$$
\hdel(\mathcal{X},\mathcal{Y})
:=\max\left\{\delta(\mathcal{X},\mathcal{Y}),\delta(\mathcal{Y},\mathcal{X})\right\},$$
where
$$
\delta(\mathcal{X},\mathcal{Y})
:=\sup_{\mathbf{x}\in\mathcal{X}:\ 
	\left\|x\right\|_{\mathcal{V}}=1}\delta(x,\mathcal{Y}),
\quad\text{with }\delta(x,\mathcal{Y}):=
\inf_{y\in\mathcal{Y}}\|x-y\|_{\mathcal{V}}.$$
We also define  $$\gamma_h:=\sup\limits_{f\in \mathcal{E}:||f ||_{2,\O}=1}
||(T-T_h)f ||_{2,\O}.$$

The following error estimates for the approximation
of eigenvalues and eigenfunctions hold true which
is obtained from Theorems~7.1 and~7.3 from \cite{BO}.

\begin{theorem}
\label{gap}
There exists a strictly positive constant $C$ such that
\begin{align}
\hdel(\mathcal{E},\mathcal{E}_h) 
& \leq C \gamma_h,\nonumber\\
\left|\mu-\mu_h^{(j)}\right|
& \le C \gamma_h \qquad \forall j=1,\ldots,m.\nonumber
\end{align}
\end{theorem}

Moreover, employing the additional regularity of the eigenfunctions,
we immediately obtain the following bound.

\begin{theorem}\label{gapr}
There exist $s> 1/2$ and $C>0$ independent of $h$ such that
\begin{align}
&||(T-T_h)f ||_{2,\O}
\le C h^{\min\{s,k-1\} }|| f||_{2,\Omega}\qquad \forall
f\in \mathcal{E},\label{bou_gamma_h}
\end{align}
and as a consequence,
\begin{align}
&\gamma_h \leq C h^{\min\{ s,k-1\} }. \label{bound1r}
\end{align}
\end{theorem}
\begin{proof}
The inequality~\eqref{bou_gamma_h} can be obtained by repeating the same
steps like in the proof of the Theorem~\ref{conv_norm_buckling} and
Lemma~\ref{LEM:REG_buck}. Estimate \eqref{bound1r} follows from the
definition of $\gamma_h$ and \eqref{bou_gamma_h}. 
\end{proof}

\begin{remark}
The error estimate obtained for the eigenpair
$(\mu,u)$ of $T$ in Theorem~\ref{gap}
implies similar estimates for the eigenpair $(\lambda:=1/\mu,u)$ of
Problem~\ref{Prob1_cont_buckling} by means of the discrete eigenvalues
$\lambda_h^{(j)}=1/\mu_h^{(j)}, 1\leq j\leq m$.
\end{remark}

Now, in what follows we will prove a double order of convergence
for the eigenvalue approximation. To prove this, we are going to 
assume that $\boldsymbol{\eta}$ is a smooth enough tensor.

\begin{theorem}\label{orderconv}
There exists a positive constant independent of $h$ such that
\begin{align*}
|\lambda - \lambda_h^{(j)}|\leq Ch^{2\min\{s,k-1 \}}\qquad \forall j=1,\ldots,m.
\end{align*}
\end{theorem}
\begin{proof}
Let $u_h\in \mathcal{E}_h$ be an eigenfunction corresponding
to one of the eigenvalues $\lambda_h^{(j)},\, j=1,\ldots,m,$
with $||u_h||_{2,\O}=1$.
From Theorem~\ref{gap}, we have that there exists
$u\in \mathcal{E}$ satisfying
\begin{align}
||u-u_h ||_{2,\O} \leq C\gamma_h.\label{cota_conv}
\end{align}

It is easy to see that from the symmetry of the bilinear forms
in the continuous and discrete
spectral problems (cf. Problem~\ref{Prob1_cont_buckling} and Problem~\ref{Prob1disc_cont_buckling}), we have 
\begin{align*}
a(u-u_{h},u-u_{h})&-\l b(u-u_{h},u-u_{h})
=a(u_{h},u_{h})-\l b(u_{h},u_{h})\\
& =a(u_{h},u_{h})-a_{h}(u_{h},u_{h})+\l_{h}^{(j)}b_{h}(u_{h},u_{h})-\l b(u_{h},u_{h})\\
& =a(u_{h},u_{h})-a_{h}(u_{h},u_{h})+(\l_{h}^{(j)}-\l)b_{h}(u_{h},u_{h})
+\l[b_h(u_{h},u_{h})-b(u_{h},u_{h})],
\end{align*}
and therefore we have the following identity
\begin{align}
(\l_{h}^{(i)}-\l )b_h(u_{h},u_{h})&=a(u-u_{h},u-u_{h})-\l b(u-u_{h},u-u_{h})\nonumber\\
&\quad+(a_{h}(u_{h},u_{h})-a(u_{h},u_{h}))+\l\left[b(u_{h},u_{h})-b_{h}(u_{h},u_{h})\right].\label{ide_tru_Padr}
\end{align}
Now, we will bound each term on the right hand
side of \eqref{ide_tru_Padr}. For the first and second term we deduce
\begin{align*}
a(u-u_{h},u-u_{h})=|u-u_h|_{2,\O}^2\le C\gamma_h^2,
\end{align*}
and
\begin{align*}
b(u-u_{h},u-u_{h})&=\int_{\Omega}\boldsymbol{\eta}\boldsymbol{\Pi}_K^{k-1}
\nabla(u-u_h)\cdot \boldsymbol{\Pi}_K^{k-1}\nabla(u-u_h)\leq
\Vert\boldsymbol{\eta}\Vert_{\infty}||u-u_h ||_{2,\Omega}^2\leq
C\gamma_h^2. 
\end{align*}
Thus, we obtain 
\begin{align}
|a(u-u_{h},u-u_{h})-\l b(u-u_{h},u-u_{h})|\leq C \gamma_h^2.\label{cot1_tru_Pa}
\end{align}
Next, to bound the third term, we consider $u_{\pi}\in L^2(\Omega)$
such that $u_{\pi}|_{K}\in \mathbb{P}_k(K) $
for all $K\in \CT_h$ and the Proposition~\ref{app1} holds true.
Hence, using the properties \eqref{consis-a} and \eqref{stab-a}
of $a_{h,K}(\cdot,\cdot )$, we have
\begin{align*}
|a_h(u_h,u_h)-a(u_h,u_h)|
&=\Big|\sum\limits_{K\in \mathcal{T}_h}\Big\{a_{h,K}(u_h-u_{\pi},u_h) -a_{K}(u_h-u_{\pi},u_h)\Big\}\Big|\nonumber\\
&\leq  \sum\limits_{K\in \mathcal{T}_h} (1+\alpha_2)a_K(u_h-u_{\pi},u_h-u_{\pi})\nonumber\\
& \leq C \sum\limits_{K\in \mathcal{T}_h}|u_h-u_{\pi}|_{2,K}^2. 
\end{align*}
Then, adding and subtracting $u$, using the triangular inequality,
Proposition~\ref{app1}
and \eqref{cota_conv}, we get 
\begin{equation}
\label{cota_2_Tru_Pa}
\left|a_h(w_h,w_h)-a(w_h,w_h)\right|
\le C\big\{ \gamma_h^2 + h^{2\min\{s,k-1 \}} \big\}.
\end{equation}
On the other hand, the fourth term can be treated as follows:
\begin{align}\label{eqregth2}
b(u_h,u_h)- b_h(u_h,u_h)
&=\sumkth \Bigg\{\int_{K}\boldsymbol{\eta}\nabla u_h\cdot \nabla u_h
-\int_{K}\boldsymbol{\eta} \boldsymbol{\Pi}_K^{k-1}\nabla u_h \cdot \boldsymbol{\Pi}_K^{k-1}\nabla u_h \Bigg\}.\nonumber\\
&=\sumkth \Bigg\{\underbrace{\int_{K}\boldsymbol{\eta}\nabla u_h\cdot (\nabla u_h-\boldsymbol{\Pi}_K^{k-1}\nabla u_h)}_{E_1}+\underbrace{\int_{K}(\nabla u_h-\boldsymbol{\Pi}_K^{k-1}\nabla u_h) \cdot \boldsymbol{\eta} \boldsymbol{\Pi}_K^{k-1}\nabla u_h}_{E_2} \Bigg\}.\nonumber
\end{align}
Now, we bound the terms $E_1$ and $E_2$. We start with $E_1$:
\begin{align*}
E_1&=\int_{K}(\boldsymbol{\eta}\nabla u_h-\boldsymbol{\Pi}_K^{k-1}(\boldsymbol{\eta}\nabla u))\cdot (\nabla u_h-\boldsymbol{\Pi}_K^{k-1}\nabla u_h)\\
&=\int_{K}\Big(\boldsymbol{\eta}\nabla u_h-\boldsymbol{\eta}\nabla u
+\boldsymbol{\eta}\nabla u-\boldsymbol{\Pi}_K^{k-1}(\boldsymbol{\eta}\nabla u)\Big)\cdot \Big(\nabla u_h-\nabla u+\nabla u -\boldsymbol{\Pi}_K^{k-1}\nabla u+\boldsymbol{\Pi}_K^{k-1}(\nabla u-\nabla u_h)\Big)\\
&\le Ch^{2\min\{s,k-1 \}},
\end{align*}
where in the last inequality we have used the triangular inequality,
the approximation properties for $\boldsymbol{\Pi}_K^{k-1}$ (cf. Lemma~\ref{app3}),
the additional regularity for the stress tensor $\boldsymbol{\eta}$ and
the additional regularity for the eigenfunctions
and finally \eqref{cota_conv} together with \eqref{bound1r}.

For the term $E_2$, we repeat the same arguments used to bound $E_1$,
we obtain that
\begin{equation}\label{eqregth2}
E_2\le Ch^{2\min\{s,k-1 \}}.
\end{equation}

On the other hand, from Problem~\ref{Prob1disc_cont_buckling},
Lemma~\ref{ha-elipt-disc} and the fact $\l_h^{(j)}\to \l$
when $h\to 0$, we have
\begin{equation*}
|b_h(u_h,u_h)|=|\frac{1}{\lambda_h^{(j)}}a_h(u_h,u_h)|\geq \frac{\alpha}{|\lambda_h^{(j)}|}||u_h ||_{2,\Omega}^2=\frac{\alpha}{|\lambda_h^{(j)}|}=C>0\label{denominador_Tru_Pa}
\end{equation*}

Thus, the proof follows from the above bound
together with estimates \eqref{ide_tru_Padr}-\eqref{eqregth2}.
\end{proof}

\section{Numerical results.}\label{SEC:Num-Res}
In this section, we report some numerical
experiments to approximate the buckling coefficients
considering different configurations of the problem,
in order to confirm the theoretical
results presented in this work for the cases $k=2$ and $k=3$.
With this purpose, we have implemented in a MATLAB code
the proposed discretization, following the arguments presented
in \cite{BBMR2014}.

To complete the construction of the discrete bilinear
form, we have taken the symmetric form $s_{K}^{D}(\cdot,\cdot)$
as the euclidean scalar product associated to the degrees of freedom,
properly scaled to satisfy \eqref{term-stab-SK} (see \cite{ABSV2016,ChM-camwa,MRV}
for further details).

On the other hand, we have tested the method by using
different families of meshes
(see Figure~\ref{FIG:Meshes_NEedges}):
\begin{itemize} 
\item $\CT_h^1$: trapezoidal meshes which consist
of partitions of the domain into $N\times N$ congruent
trapezoids, all similar to the trapezoid with
vertices $(0,0)$, $(1/2,0)$, $(1/2,2/3)$ and $(0,1/3)$;
\item $\CT_h^2$: hexagonal meshes;
\item $\CT_h^3$:  triangular meshes;
\item $\CT_h^4$:  distorted concave rhombic quadrilaterals.
\end{itemize}

We have used successive refinements of an initial
mesh (see Figure~\ref{FIG:Meshes_NEedges}).
The refinement parameter $N$ used to label
each mesh is the number of elements
on each edge of the plate.

\begin{figure}[!t]
\begin{center}
\begin{minipage}{6cm}
\centering\includegraphics[height=6cm, width=6cm]{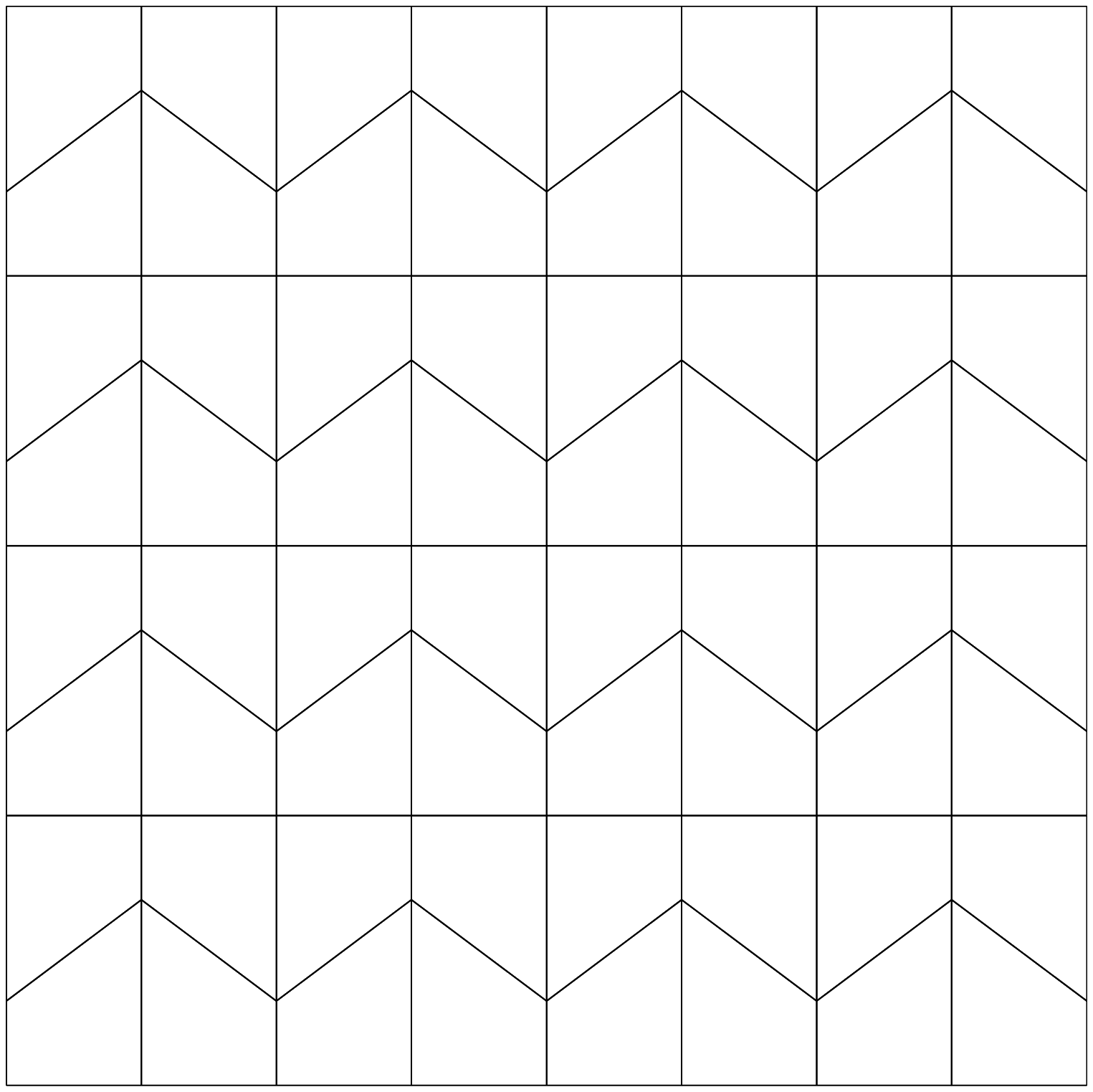}
\end{minipage}
\begin{minipage}{6cm}
\centering\includegraphics[height=6cm, width=6cm]{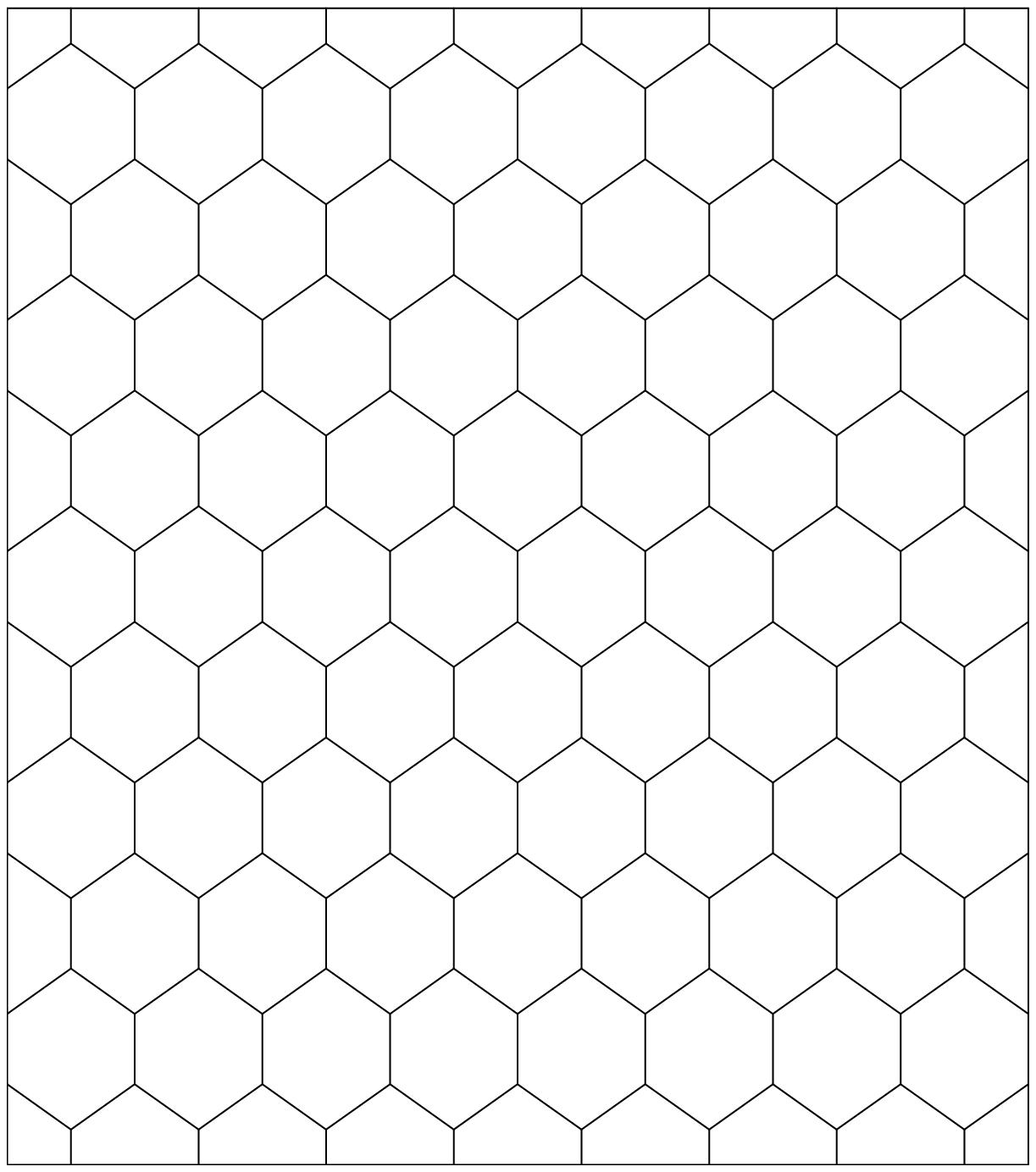}
\end{minipage}
\begin{minipage}{6cm}
\centering\includegraphics[height=6cm, width=6cm]{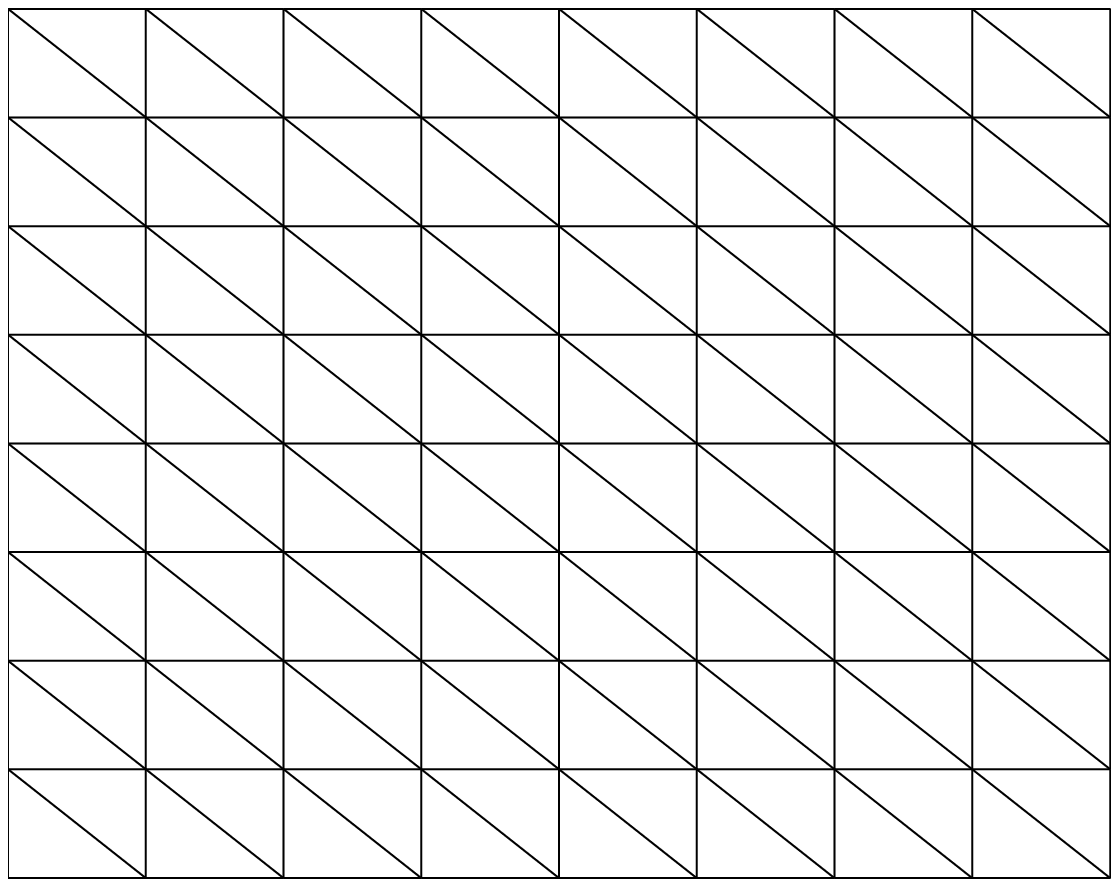}
\end{minipage}
\begin{minipage}{6cm}
\centering\includegraphics[height=6cm, width=6cm]{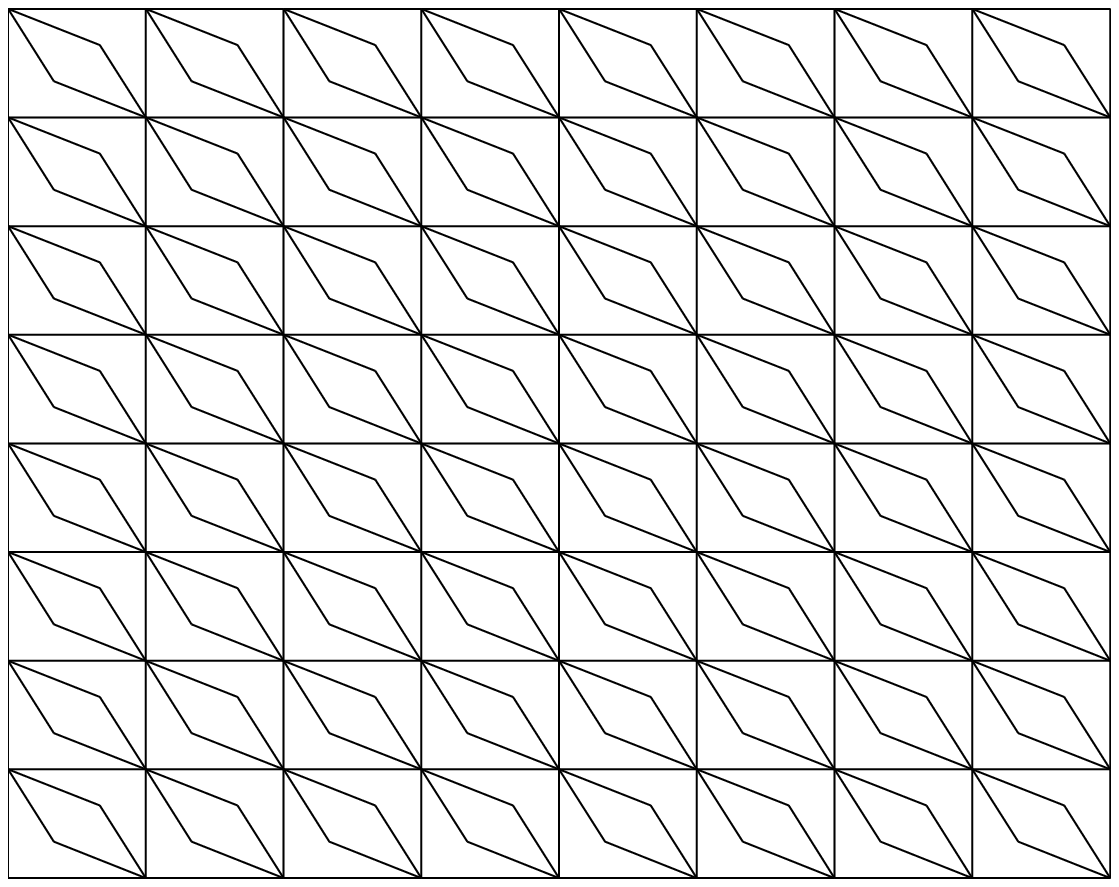}
\end{minipage}
\caption{ Sample meshes: $\CT_h^1$ (top left), $\CT_h^2$ (top right),
$\CT_h^3$ (bottom left) and $\CT_h^{4}$ (bottom right), for $N=8$.}
\label{FIG:Meshes_NEedges}
\end{center}
\end{figure}

We have chosen two configurations for the computational domain $\Omega$:
$\Omega_S:=(0,1)\times(0,1)$ and $\Omega_L:=(0,1)\times(0,1)\backslash [1/2,1)\times [1/2,1)$.
Even though our theoretical analysis has been developed only for clamped plates,
we will consider in Section~\ref{Sec:SF_buckling} other boundary conditions.

In order to compare our results for the buckling problem,
we introduce a non-dimensional buckling coefficient, which is
defined as:
\begin{equation}\label{scaled}
\widehat{\lambda}_{h}^{(j)}:=\frac{\lambda_{h}^{(j)} L}{\pi^2},
\end{equation}
where $L$ is the plate side length.

Moreover, we will consider different in-plane compressive
stress $\boldsymbol{\eta}$.
More precisely, we will compute the non-dimensional
buckling coefficients using the following $\boldsymbol{\eta}$: 
\begin{equation*}
\boldsymbol{\eta}_1:=\begin{pmatrix}1&0\\0&1\end{pmatrix}, \qquad
\boldsymbol{\eta}_2:=\begin{pmatrix}1&0\\0&0\end{pmatrix}, \qquad
\boldsymbol{\eta}_3:=\begin{pmatrix}0&1\\1&0\end{pmatrix}.
\end{equation*}

The physical meaning of the tensors $ \boldsymbol{\eta}_1, \boldsymbol{\eta}_2$
and $\boldsymbol{\eta}_3$ is illustrated in
Figures~\ref{eta_meaning} and \ref{eta_meaning2}, respectively.
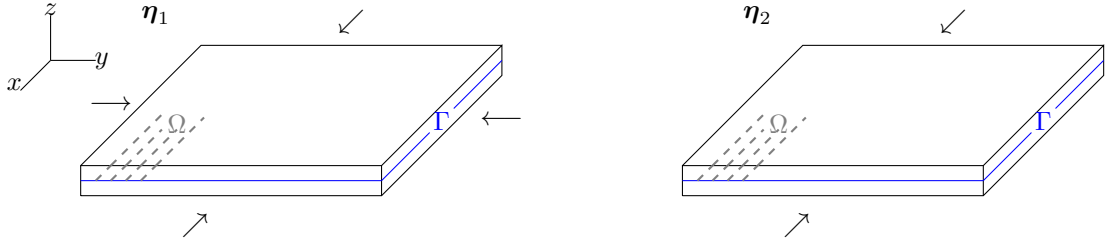
\begin{figure}
\begin{center}
\begin{tikzpicture}[scale=0.4]
\draw (0,0)--(10,0)--(14,4)--(14,5)--(10,1)--(0,1)--(0,0);
\draw (0,1)--(4,5)--(14,5);
\draw (10,0)--(10,1);
\draw[blue] (0,0.5)--(10,0.5);
\draw[blue] (10,0.5)--(11.6,2.1);
\draw[blue] (12.4,2.9)--(14,4.5);
\node at (12,2.5){${\textcolor{blue}{\Gamma}}$};
\draw[gray,thick,dashed] (0.5,0.5)--(1,1); 
\draw[gray,thick,dashed] (1,0.5)--(1.5,1);
\draw[gray,thick,dashed] (1.5,0.5)--(2,1); 
\draw[gray,thick,dashed] (2,0.5)--(2.5,1); 
\draw[gray,thick,dashed] (1,1)--(2.8,2.8);
\draw[gray,thick,dashed] (1.5,1)--(2.7,2.2);
\draw[gray,thick,dashed] (2,1)--(3.2,2.2);
\draw[gray,thick,dashed] (2.5,1)--(4.1,2.6);
\node at (3.2,2.4){$\textcolor{gray}{\Omega}$};
\node at (2.5,6) {$\boldsymbol{\eta}_1$}; 
\node at (3.8,-1) {$\nearrow$};
%
\node at (14,2.5){$\longleftarrow$};
\node at (9,5.8){$\swarrow$};
\node at (1,3) {$\longrightarrow$};
\draw (-1,4.5)--(-1,6); \node at (-1,6.2) {$z$};
\draw (-1,4.5)--(0.5,4.5); \node at (0.7,4.5) {$y$};
\draw (-1,4.5)--(-2,3.5); \node at (-2.2,3.7) {$x$};
\draw (20,0)--(30,0)--(34,4)--(34,5)--(30,1)--(20,1)--(20,0);
\draw (20,1)--(24,5)--(34,5);
\draw (30,0)--(30,1);
\draw[blue] (20,0.5)--(30,0.5);
\draw[blue] (30,0.5)--(31.6,2.1);
\draw[blue] (32.4,2.9)--(34,4.5);
\node at (32,2.5){${\textcolor{blue}{\Gamma}}$};
\draw[gray,thick,dashed] (20.5,0.5)--(21,1); 
\draw[gray,thick,dashed] (21,0.5)--(21.5,1);
\draw[gray,thick,dashed] (21.5,0.5)--(22,1); 
\draw[gray,thick,dashed] (22,0.5)--(22.5,1); 
\draw[gray,thick,dashed] (21,1)--(22.8,2.8);
\draw[gray,thick,dashed] (21.5,1)--(22.7,2.2);
\draw[gray,thick,dashed] (22,1)--(23.2,2.2);
\draw[gray,thick,dashed] (22.5,1)--(24.1,2.6);
\node at (23.2,2.4){$\textcolor{gray}{\Omega}$};
\node at (22.5,6) {$\boldsymbol{\eta}_2$}; 
\node at (23.8,-1) {$\nearrow$};
\node at (29,5.8){$\swarrow$};
\end{tikzpicture}
\caption{$\boldsymbol{\eta}_1$ (left) correspond to a uniformly compressed plate
(in the $x$, $y$ directions) and
$\boldsymbol{\eta}_2$ (right) correspond
to a plate subjected to uniaxial compression (in the $x$ direction).}\label{eta_meaning}
\end{center} 
\end{figure}
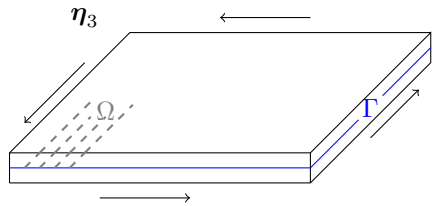
\begin{figure}
\begin{center}
\begin{tikzpicture}[scale=0.4]
\draw (0,0)--(10,0)--(14,4)--(14,5)--(10,1)--(0,1)--(0,0);
\draw (0,1)--(4,5)--(14,5);
\draw (10,0)--(10,1);
\draw[blue] (0,0.5)--(10,0.5);
\draw[blue] (10,0.5)--(11.6,2.1);
\draw[blue] (12.4,2.9)--(14,4.5);
\node at (12,2.5){${\textcolor{blue}{\Gamma}}$};
\draw[gray,thick,dashed] (0.5,0.5)--(1,1); 
\draw[gray,thick,dashed] (1,0.5)--(1.5,1);
\draw[gray,thick,dashed] (1.5,0.5)--(2,1); 
\draw[gray,thick,dashed] (2,0.5)--(2.5,1); 
\draw[gray,thick,dashed] (1,1)--(2.8,2.8);
\draw[gray,thick,dashed] (1.5,1)--(2.7,2.2);
\draw[gray,thick,dashed] (2,1)--(3.2,2.2);
\draw[gray,thick,dashed] (2.5,1)--(4.1,2.6);
\node at (3.2,2.4){$\textcolor{gray}{\Omega}$};
\draw[->] (3.0,-0.5) -- (6.0,-0.5);
\draw[->] (12,1.5)--(13.6,3.1);
\draw[<-] (7,5.5)--(10,5.5);
\draw[<-] (0.5,2)--(2.5,4);
\node at (2.5,5.5) {$\boldsymbol{\eta}_3$}; 
\end{tikzpicture}
\caption{$\boldsymbol{\eta}_3$ 
correspond to a plate subjected to shear load.}\label{eta_meaning2}
\end{center} 
\end{figure}

\subsection{Clamped square plate.}
\label{Sec:CC_buckling}

In this numerical test we compute the non-dimensional buckling coefficients (cf. \eqref{scaled})
for a uniformly compressed square plate $\Omega_S$.
This corresponds to the stress field $\boldsymbol{\eta}_1$

We report in Table~\ref{NVB_I_CC_O_sq_k=2}
the four lowest non-dimensional buckling coefficients computed
with the virtual element method analyzed in this paper.
The polynomial degrees are given by $k=2,3$ and
with two different families of meshes and $N = 32,64,128$.
The table includes orders of convergence as well as accurate values
extrapolated by means of a least-squares fitting. In the last row of the table,
we show the values obtained by extrapolating those computed with
different method presented in \cite{MoRo2009}.

\begin{table}[ht]
\caption{Lowest non-dimensional buckling coefficients $\widehat{\lambda}_{h}^{i}$, $i=1,2,3,4$
of a clamped square plate subjected to a plane
stress field  $\boldsymbol{\eta}_1$.}
\label{NVB_I_CC_O_sq_k=2}
\begin{center}
{\small\begin{tabular}{lllcccc}
\hline
\hline 
Mesh & $k$ & $N$ &$\widehat{\lambda}_{h}^{1}$ & $\widehat{\lambda}_{h}^{2}$
& $\widehat{\lambda}_{h}^{3}$ & $\widehat{\lambda}_{h}^{4}$ \\	
\hline
\hline
&& $32$ &5.2724&    9.1716&    9.2744&   12.8252\\
&& $64$ &5.2952&    9.2906&    9.3174&   12.9461\\
$\CT_h^2$&2& $128$&5.3014&    9.3229&    9.3297&   12.9786\\
&& Order    &1.86  &    1.88  &    1.80  &    1.89  \\
&& Extrap.  &5.3038&    9.3350&    9.3347&   12.9907\\
\hline
&& $32$ & 	 5.3037&    9.3345&    9.3347&   12.9918\\
&& $64$ &    5.3036&    9.3342&    9.3342&   12.9904\\
$\CT_h^2$&3& $128$&    5.3036&    9.3342&    9.3342&   12.9904\\
&& Order  &      3.95  &    3.95  &    3.94  &    3.93  \\
&& Extrap.&      5.3036&    9.3342&    9.3342&   12.9903\\
\hline
\hline 
		& & $32$  &5.3192&    9.3581&    9.3968&   13.0934\\
		& & $64$ &5.3075&    9.3401&    9.3498&   13.0162\\
$\CT_h^4$&$2$& $128$ &5.3046&    9.3356&    9.3381&   12.9968\\
		& & Order     &2.00  &    2.00  &    2.00  &    1.99  \\
		& & Extrap.   &5.3036&    9.3342&    9.3341&   12.9903\\
\hline
		& & $32$  &5.3039&    9.3348&    9.3353&   12.9939\\
		& & $64$ &5.3036&    9.3342&    9.3342&   12.9906\\
$\CT_h^4$&$3$& $128$ &5.3036&    9.3342&    9.3342&   12.9904\\
		& & Order	 &3.94  &    3.93  &    3.93  &    3.91  \\
		& & Extrap.	 &5.3036&    9.3342&    9.3342&   12.9903\\
\hline\rowcolor{Gray}
\cite{MoRo2009}&&& 5.3037& 9.3337 & 9.3337 & 12.9909\\\hline
\end{tabular}}
\end{center}
\end{table}

In this case, since $\Omega_S$ is convex, the problem have smooth eigenfunctions,
as a consequence, when using degree $k$, the order of convergence
is $2(k-1)$ as the theory predicts (cf. Theorem~\ref{orderconv}).
Moreover, the results obtained by the two methods agree perfectly well.

In the next test we compute once again the non-dimensional
buckling coefficients (in absolute value) of the same plate as in the previous example,
subjected to a uniform shear load. This corresponds to the stress field $\boldsymbol{\eta}_3$.

In Table~\ref{NVB_J_CC_O_sq_k=2} we report
the four lowest non-dimensional buckling coefficients (in absolute value) considering the
stress field $\boldsymbol{\eta}_3$.
Once again, the polynomial degrees are given by $k=2,3$ and
with two different families of meshes and $N=32,64,128$.
The table includes orders of convergence as well as accurate values
extrapolated by means of a least-squares fitting.
In the last row of the table, we show the values obtained
by extrapolating those computed with
different method presented in \cite{MoRo2009}.

Once again, it can be clearly observed from Table~\ref{NVB_J_CC_O_sq_k=2}
that our method computes the scaled buckling coefficients (cf.\eqref{scaled}) with an
optimal order of convergence and that the agreement
with the method from \cite{MoRo2009} is excellent.

\begin{table}[ht]
\caption{Lowest non-dimensional buckling coefficients (in absolute value)
$\widehat{\lambda}_{h}^{i}$, $i=1,2,3,4$
of a clamped square plate subjected to a plane stress tensor field  $\boldsymbol{\eta}_3$.}
\label{NVB_J_CC_O_sq_k=2}
\begin{center}
{\small\begin{tabular}{lllcccc}
\hline
\hline 
Mesh & $k$ & $N$  &$\widehat{\lambda}_{h}^{1}$ & $\widehat{\lambda}_{h}^{2}$
& $\widehat{\lambda}_{h}^{3}$ & $\widehat{\lambda}_{h}^{4}$ \\	
\hline
&& $32$ &14.6083&   16.8405&   33.2148&   35.2101\\
&& $64$ &14.6331&   16.8983&   33.3053&   35.2700\\
$\CT_h^1$&2& $128$&14.6397&   16.9137&   33.3319&   35.2888\\
&& Order    &1.89   &   1.92   &   1.77   &   1.67  \\
&& Extrap.  &14.6422&   16.9191&   33.3429&   35.2974\\
\hline 
&& $32$ &   14.6470&   16.9242&   33.3795&   35.3423\\
&& $64$ &   14.6423&   16.9192&   33.3437&   35.2986\\
$\CT_h^1$&3& $128$&   14.6420&   16.9189&   33.3413&   35.2957\\
&& Order  &      3.93  &    3.95  &    3.90  &    3.89  \\
&& Extrap.&     14.6420&   16.9188&   33.3411&   35.2954\\
\hline 
\hline
   &      & $36$ &14.6330&   16.9071&   33.2479&   35.1883\\
	&	 & $64$ &14.6398&   16.9159&   33.3178&   35.2685\\
$\CT_h^3$&2& $128$&14.6415&   16.9181&   33.3353&   35.2887\\
		& & Order    &2.02   &   2.01   &   1.99   &   1.99  \\
		& & Extrap.  &14.6420&   16.9188&   33.3412&   35.2955\\
\hline 
		 && $32$ &   14.6455&   16.9214&   33.3536&   35.3098\\
		 && $64$ &   14.6423&   16.9190&   33.3421&   35.2966\\
$\CT_h^3$&3 &$128$&   14.6420&   16.9189&   33.3412&   35.2955\\
		 & &Order  	&    3.63  &    3.71  &    3.65  &    3.62  \\
		 & & Extrap.	&   14.6420&   16.9188&   33.3411&   35.2954\\
\hline\rowcolor{Gray}
\cite{MoRo2009}&& &14.6420 & 16.9195 & 33.3376 & -\\\hline
\end{tabular}}
\end{center}
\end{table}

We show in Figure~\ref{FIG:NVB_J_CC_O_sq_k=2_Estabilizado}
the buckling mode associated with the lowest scaled buckling coefficient.

\begin{figure}[ht]
\begin{center}
\begin{minipage}{5cm}
\centering\includegraphics[height=5cm, width=6cm]{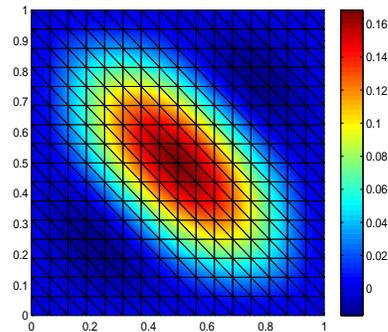}
\end{minipage}
\caption{Test 1. Buckling mode associated to the first non-dimensional buckling coefficient
of a square plate subjected to a plane stress tensor field $\boldsymbol{\eta}_3$.}
\label{FIG:NVB_J_CC_O_sq_k=2_Estabilizado}
\end{center}
\end{figure}

\subsection{Clamped L-shaped plate.}
\label{Sec:CC_buckling_L}

In this numerical test, we consider an L-shaped domain: $\Omega_L$.
We have used triangular and concave meshes as those
shown in $\CT_h^3$ and $\CT_h^4$, respectively (see Figure~\ref{FIG:Meshes_NEedges}).
Once again, the refinement parameter $N$ is the number of elements on each edge.

Table~\ref{NVB_I_CC_O_L_k=2} reports the four lowest non-dimensional buckling coefficient
computed with the method analyzed in this paper with polynomial degree $k=2$.
We include in this table orders of convergence,
as well as accurate values extrapolated by means
of a least-squares fitting again.  
In the last row of the table, we show the values obtained
by extrapolating those computed with
different method presented in \cite{MoRo2009}.

\begin{table}[ht]
\caption{Four lowest non-dimensional buckling coefficient of a clamped L-shaped plate and
subjected to a plane stress tensor field  $\boldsymbol{\eta}_1$.}
\label{NVB_I_CC_O_L_k=2}
\begin{center}
{\small\begin{tabular}{lllcccc}
\hline
\hline 
Mesh & $k$ &$N$& $\widehat{\lambda}_{1h}$ & $\widehat{\lambda}_{2h}$ & $\widehat{\lambda}_{3h}$ & $\widehat{\lambda}_{4h}$ \\	
\hline
\hline 
&& $32$ &13.1749&   15.0809&   17.0798&   19.9445\\
&& $64$&13.0847&   15.0234&   17.0203&   19.8758\\
$\CT_h^{3}$ &2& $128$&13.0495&   15.0083&   17.0042&   19.8582\\
&& Order	   &1.36   &   1.93   &   1.89   &   1.97   \\
&& Extrap.  &13.0271&   15.0029&   16.9983&   19.8522\\
\hline
&&$32$&13.1949 &  15.1399&   17.1801&   20.1590\\
&&$64$&13.0903 &  15.0388&   17.0453&   19.9297\\
$\CT_h^4$&2&$128$&13.0511 &  15.0124&   17.0105&   19.8717\\
&&Order&1.41    &  1.94   &   1.95   &   1.98  \\
&&Extrap.&13.0274 &  15.0031&   16.9983&   19.8519\\
\hline\rowcolor{Gray}
&&\cite{MoRo2009}&13.0290&15.0036&16.9949&-\\\hline 
\end{tabular}}
\end{center}
\end{table}

We observe that for the lowest non-dimensional buckling coefficient,
the method converges with order close to $1.089$, which is the expected one
because of the singularity of the solution (see \cite{G}).
For the other non-dimensional buckling coefficients, the method
converges with larger orders.

We show in Figure~\ref{FIG:NVB_I_CC_O_L_k=2_Estabilizado}
the buckling mode associated with the lowest scaled buckling coefficient.

\begin{figure}[ht]
\begin{center}
\begin{minipage}{5cm}
\centering\includegraphics[height=5cm, width=6cm]{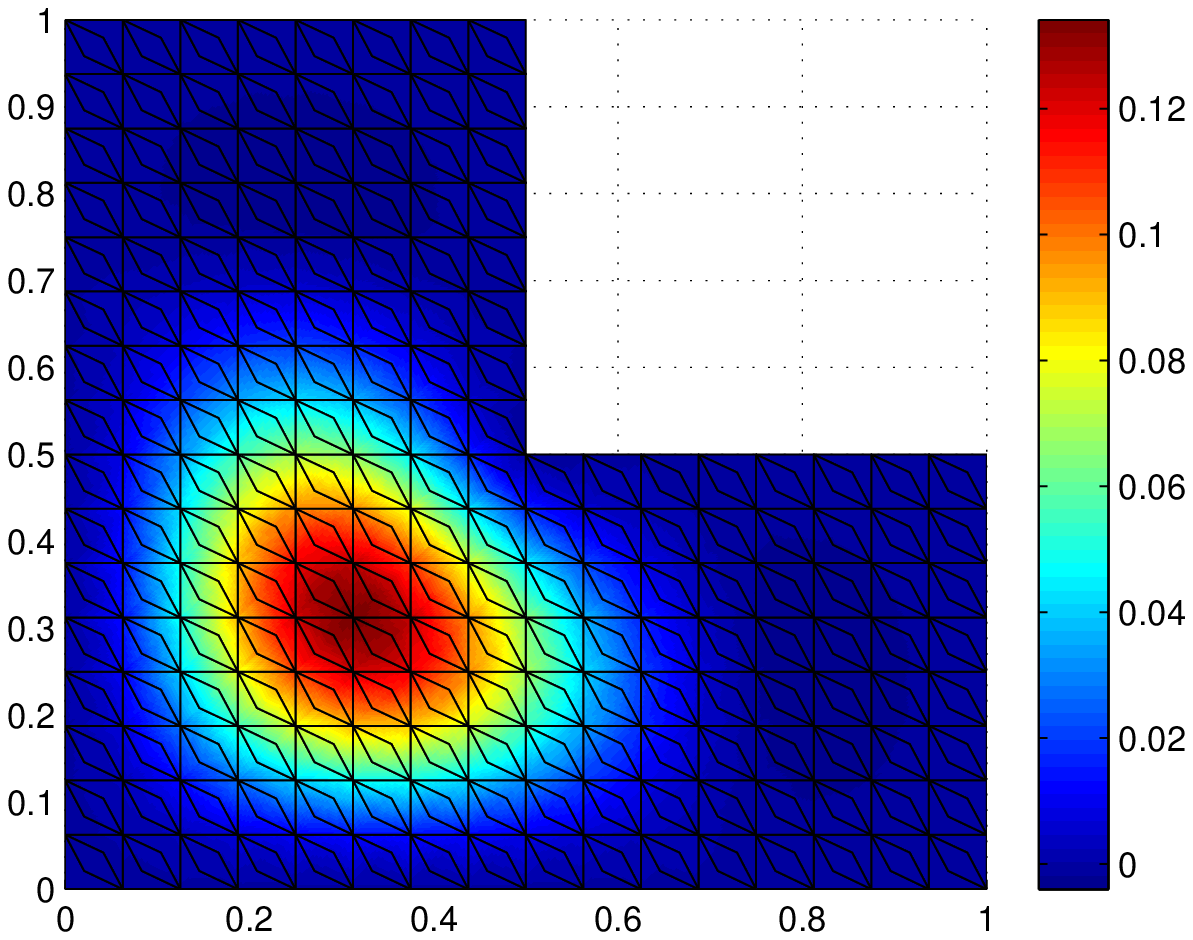}
\end{minipage}
\caption{Test 2. Buckling mode associated to the first non-dimensional buckling coefficient
of a clamped L-shaped plate subjected to a plane stress tensor field $\boldsymbol{\eta}_1$.}
\label{FIG:NVB_I_CC_O_L_k=2_Estabilizado}
\end{center}
\end{figure}

\subsection{Simply supported-free square plate.}
\label{Sec:SF_buckling}

In this final test, which is not covered by our
theory since our theoretical results has been
developed only for clamped plates,
we have computed the non-dimensional
buckling coefficient of a simply supported-free
square plate, subjected to linearly varying
in-plane load in one direction ($x$ direction).
This corresponds to a plane stress field given by
\begin{equation}\label{eta_variable}
\widetilde{\boldsymbol{\eta}}_2:=\begin{pmatrix}
1-\alpha \frac{y}{L}   & 0\\0&0
\end{pmatrix},
\end{equation}
with values of $\alpha$ in $\{0,2/3,1,4/3,2 \}$.
We observe that for $\alpha=0$, we obtain the plane
stress tensor field $\boldsymbol{\eta}_2$.

We take an square plate $\Omega_S$ which
has two simply supported edges and two free edges.

We report in Table~\ref{ClamFree_alfa_k=2}
the non-dimensional buckling coefficient.
The polynomial degrees are given by $k=2,3$ and
the family of meshes $\CT_h^2$ with $N=32,64,128$.
The table includes computed orders of convergence and
extrapolated more accurate values of each eigenvalue
obtained by means of a least-squares fitting.

\begin{table}[ht]
\caption{Non-dimensional buckling coefficient $\widehat{\lambda}_{1h}$ 
for different values of $\alpha$ of a
square plate with mixed boundary conditions and subjected to
linearly varying in-plane load in one direction $\widetilde{\boldsymbol{\eta}}_2$.}
\label{ClamFree_alfa_k=2}
\begin{center}
{\small\begin{tabular}{lllccccc}
\hline
\hline 
Mesh & $k$ & $N$ &$\alpha=0$ & $\alpha=2/3$ & $\alpha=1$ & $\alpha=4/3$ & $\alpha=2$ \\	
\hline
\hline 
		 & &32     &0.9984&    1.4474&    1.7763&  2.1687&  3.0676\\
		 & &64     &0.9996&    1.4490&    1.7782&  2.1709&  3.0702\\
$\CT_h^2$&2&128    &0.9999&    1.4495&    1.7787&  2.1715&  3.0710\\
		 & &Order  &1.91  &    1.90  &    1.90  &   1.88 &  1.85  \\
		 & &Extrap.&1.0000&    1.4496&    1.7789&  2.1717&  3.0713\\		 
\hline 
		 & &32     &1.0000&    1.4496&    1.7789&  2.1717&  3.0712\\
		 & &64     &1.0000&    1.4496&    1.7789&  2.1717&  3.0712\\
$\CT_h^2$&3&128	   &1.0000&    1.4496&    1.7789&  2.1717&  3.0712\\
		 & &Order  &4.00  &    4.00  &    4.00  &   4.00 &  4.00  \\
		 & &Extrap.&1.0000&    1.4496&    1.7789&  2.1717&  3.0712\\
\hline
\end{tabular}}
\end{center}
\end{table}

It can be clearly observed from Table~\ref{ClamFree_alfa_k=2}
that the proposed virtual scheme computes the scaled buckling coefficient
(cf. \eqref{scaled}) with an optimal order of convergence for all the values
of $\alpha$.

Finally, we show in Figure~\ref{FIG:Square01_n=16_hola}
the buckling mode associated with the lowest scaled buckling coefficient
for different values of the parameter $\alpha$.

\begin{figure}[ht]
\begin{center}
\begin{minipage}{16cm}
\centering\includegraphics[height=4cm, width=5cm]{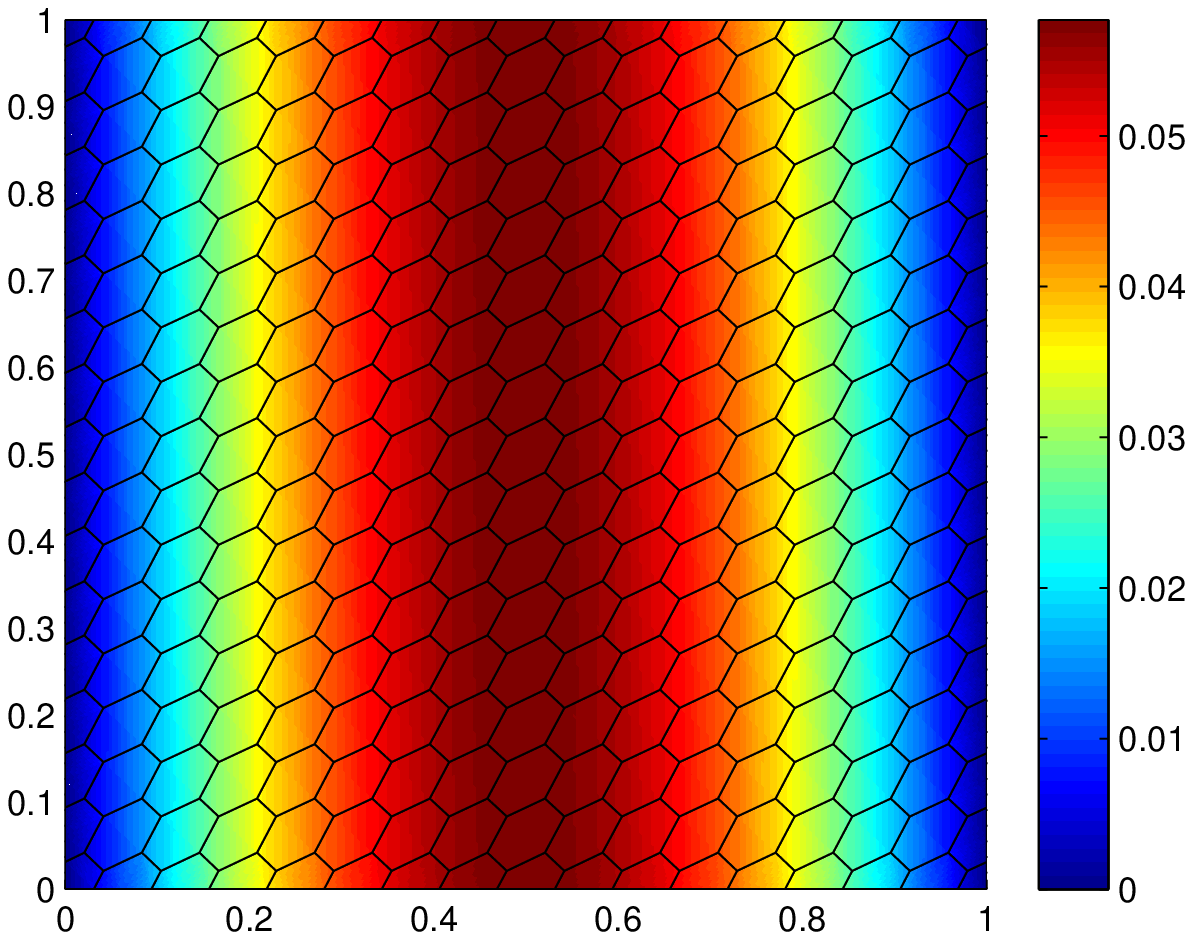}
\centering\includegraphics[height=4cm, width=5cm]{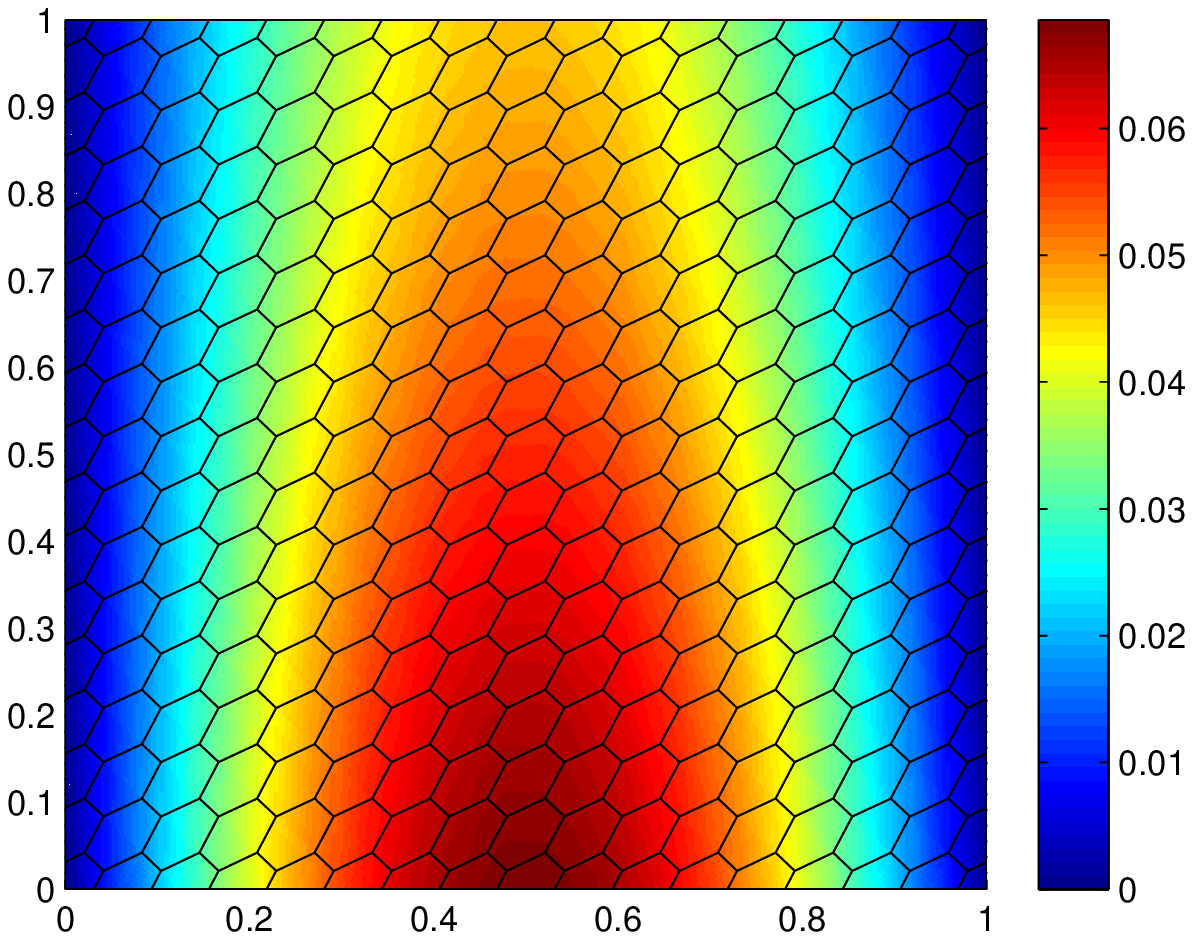}
\centering\includegraphics[height=4cm, width=5cm]{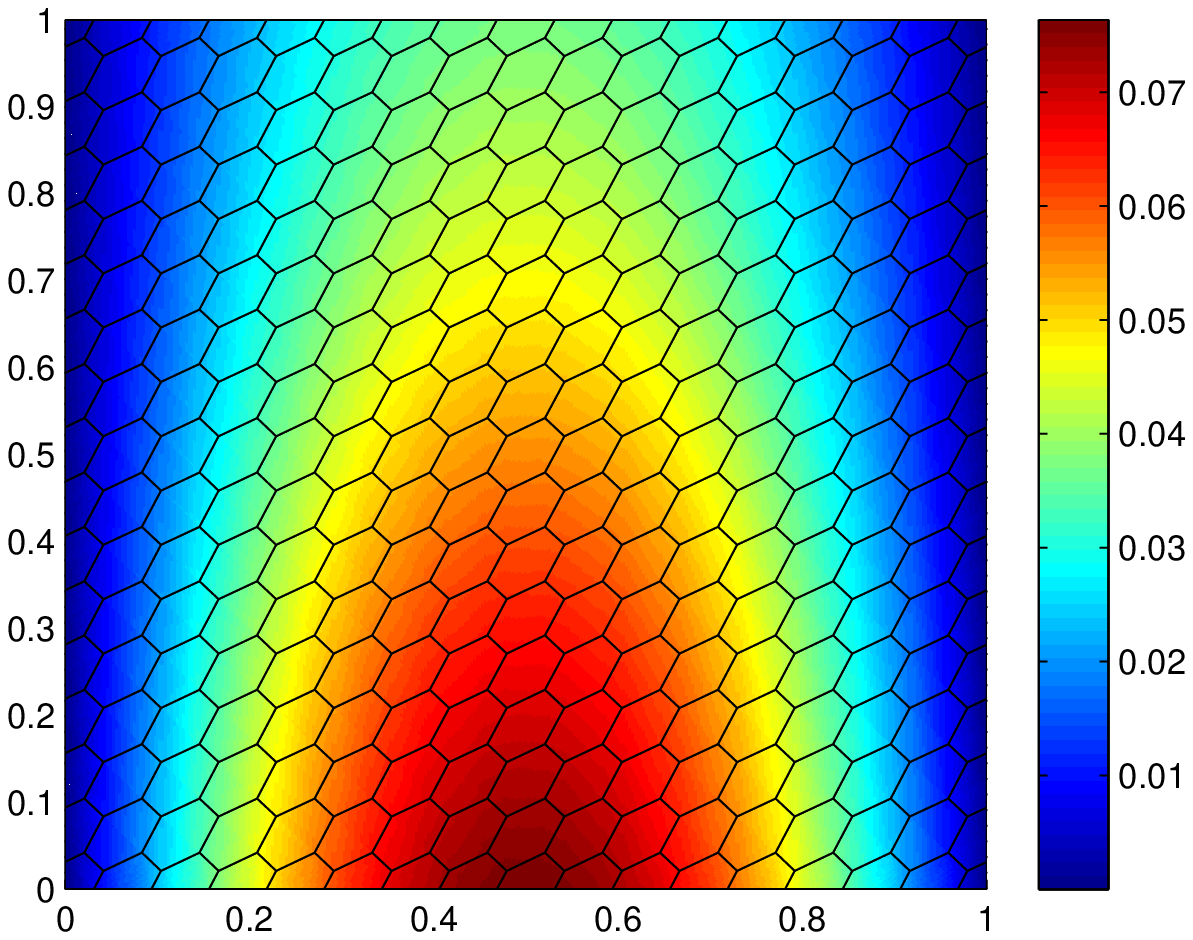}
\end{minipage}
\begin{minipage}{16cm}
\centering\includegraphics[height=4cm, width=5cm]{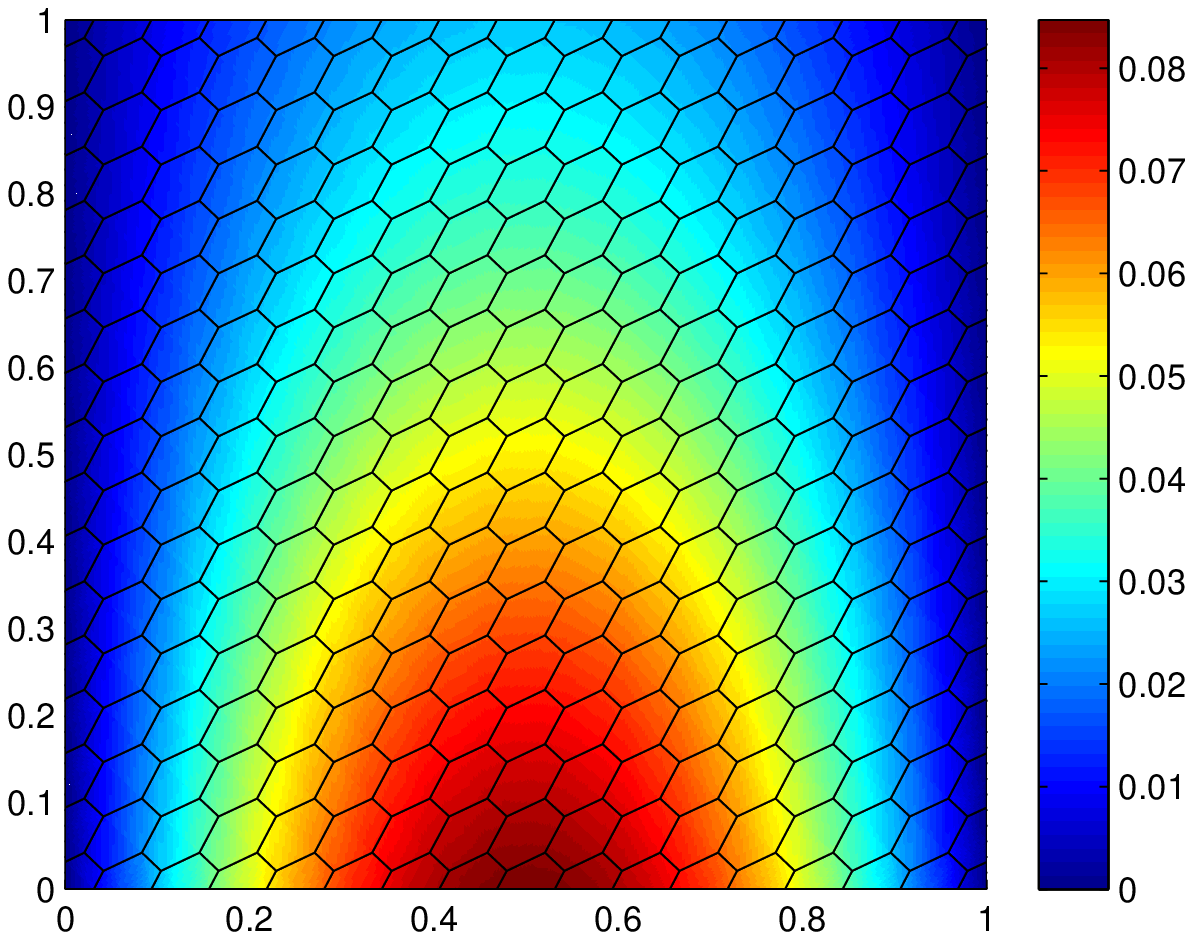}
\centering\includegraphics[height=4cm, width=5cm]{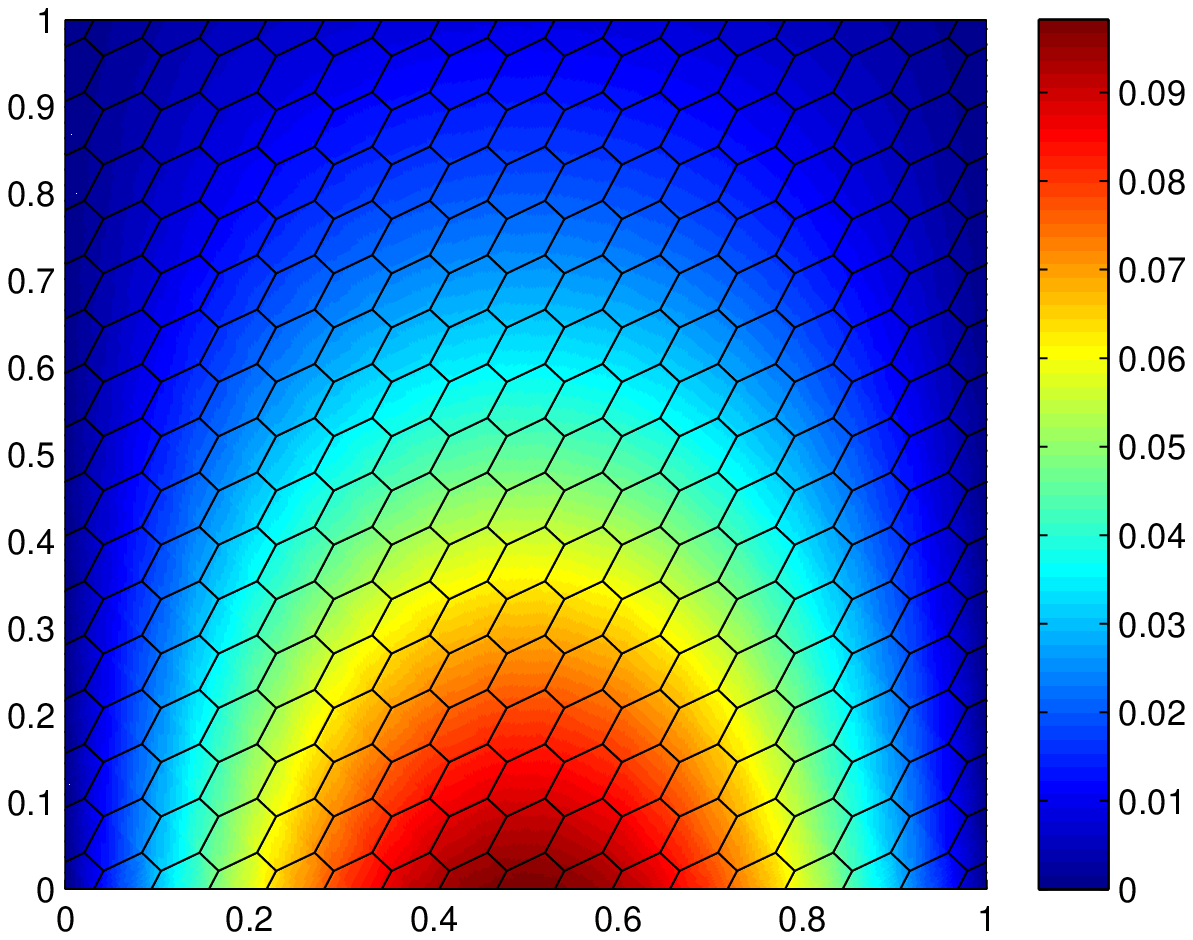}
\end{minipage}
\caption{Test 3. Buckling modes associated to the first non-dimensional buckling coefficient
$\widehat{\lambda}_{1h}$ of a square plate with mixed boundary conditions and subjected to
linearly varying in-plane load in one direction $\widetilde{\boldsymbol{\eta}}_2$ (cf. \eqref{eta_variable}):
$\alpha=0$ (top left), $\alpha=2/3$ (top middle),
$\alpha=1$ (top right), $\alpha=4/3$ (bottom left), $\alpha=2$ (bottom right).}
\label{FIG:Square01_n=16_hola}
\end{center}
\end{figure}

\section*{Acknowledgments}

The First author was partially supported by
CONICYT-Chile through FONDECYT project 1180913 and by project AFB170001
of the PIA Program: Concurso Apoyo a Centros Cient\'ificos y Tecnol\'ogicos
de Excelencia con Financiamiento Basal.
The second author was partially supported by a CONICYT-Chile
fellowship.

\bibliographystyle{amsplain}

\end{document}